\documentclass[11pt,reqno]{article}

\usepackage{amsthm}
\usepackage{amsmath}
\usepackage{amsfonts}
\usepackage{amssymb}
\usepackage{fullpage}

\newtheorem{theirtheorem}{Theorem}

\theoremstyle{plain}
\newtheorem{thm}{\textbf{Theorem}}[section]
\newtheorem{lem}[thm]{\textbf{Lemma}}

\newtheorem{claim}{\textbf{Claim}}

\newtheorem*{Freiman-3k-4}{\textbf{Freiman $3k-4$ Theorem}}
\newtheorem*{Freiman-3k-3}{\textbf{Freiman $3k-3$ Theorem}}

\newcommand{\Sum}[2]{\underset{#1}{\overset{#2}{\sum}}}

\newcommand{\Z}{\mathbb{Z}}

\newcommand{\be}{\begin{equation}}
\newcommand{\ee}{\end{equation}}

\newcommand{\conv}{\mbox{Conv}}
\newcommand{\C}{\mathbf{C}}
\newcommand{\R}{\mathbb{R}}

\newcommand{\ber}{\begin{eqnarray}}
\newcommand{\eer}{\end{eqnarray}}
\newcommand{\nn}{\nonumber}

\begin{document}

\title{Properties of Two Dimensional Sets with Small Sumset}
\author{David Grynkiewicz
\thanks{Supported in part by the National Science
  Foundation, as an MPS-DRF postdoctoral fellow, under grant DMS-0502193.}
 \and Oriol Serra\thanks{Supported by the
Spanish Research Council under project
  MTM2005-08990-C02-01 and by the Catalan Research Council under
  project 2005SGR00256}}

\maketitle

\begin{abstract}Let $A,\,B\subseteq \mathbb{R}^2$ be finite,
nonempty subsets, let $s\geq 2$ be an integer, and let $h_1(A,B)$
denote the minimal number $t$ such that there exist $2t$ (not
necessarily distinct) parallel lines,
$\ell_1,\ldots,\ell_{t},\ell'_1,\ldots,\ell'_{t}$, with $A\subseteq
\bigcup_{i=1}^{t}\ell_i$ and $B\subseteq\bigcup_{i=1}^{t}\ell'_i$.
Suppose $h_1(A,B)\geq s$. Then we show that:

(a) if $||A|-|B||\leq s$ and $|A|+|B|\geq 4s^2-6s+3$, then
$$|A+B|\geq (2-\frac 1 s)(|A|+|B|)-2s+1;$$

(b) if $|A|\geq |B|+s$ and $|B|\geq 2s^2-\frac{7}{2}s+\frac{3}{2}$,
then
$$|A+B|\geq |A|+(3-\frac 2 s)|B|-s;$$

(c) if $|A|\geq \frac{1}{2}s(s-1)|B|+s$ and either $|A|>
\frac{1}{8}(2s-1)^2|B|-\frac{1}{4}(2s-1)+\frac{(s-1)^2}{2(|B|-2)}$
or $|B|\geq \frac{2s+4}{3}$, then
$$|A+B|\geq |A|+s(|B|-1).$$
This extends the $2$-dimensional case of the Freiman $2^d$--Theorem
to distinct sets $A$ and $B$, and, in the symmetric case $A=B$,
improves the best prior known bound for $|A|+|B|$ (due to
Stanchescu, and which was cubic in $s$) to an  exact
value.

As part of the proof, we give general lower bounds for two
dimensional subsets that improve the $2$-dimensional case of
estimates of Green and Tao and of Gardner and Gronchi, and that
generalize the $2$-dimensional case of the Brunn-Minkowski Theorem.

\end{abstract}

\section{Introduction}

\indent\indent Given a pair of finite subsets $A$ and $B$ of an
abelian group $G$, their Minkowski sum, or simply sumset, is
$A+B=\{a+b\mid a\in A,\,b\in B\}$. Furthermore, if $G=\mathbb{R}^d$
and $H$ is a subspace, then we let $\phi_H:\R^d\rightarrow \R^d/H$
denote the natural projection modulo $H$, and we let $h_{d-1}(A,B)$
be the minimal number $s$ such that there exist $2s$ (not
necessarily distinct) parallel hyperplanes,
$H_1,\ldots,H_{s},H'_1,\ldots,H'_{s}$, with $A\subseteq
\bigcup_{i=1}^{s}H_i$ and $B\subseteq\bigcup_{i=1}^{s}H'_i$.
Alternatively, $h_{d-1}(A,B)$ is the minimal $s$ such that there
exists a $(d-1)$--dimensional subspace $H$ with
$|\phi_H(A)|,\,|\phi_H(B)|\leq s$.

It is the central goal of inverse additive theory to describe the
structure of sumsets and their summands. One of the most classical
results is the Freiman $2^d$--Theorem \cite{F1} \cite{bilu}
\cite{Natbook} \cite{taobook}, which says that a subset of $\R^d$
with small sumset must be contained in a small number of parallel
hyperplanes.

\begin{theirtheorem}[Freiman $2^d$--Theorem]\label{F2d-theorem} Let
$d\geq 2$ be an integer and let $0<c<2^d$. There exist constants
$k=k(c,d)$ and $s=s(c,d)$ such that if $A\subseteq \R^d$ is a
finite, nonempty subset satisfying $|A|\geq k$ and $|A+A|< c|A|$,
then $h_{d-1}(A,A)< s$.
\end{theirtheorem}

From the pigeonhole principle, one then easily infers there must
exist a hyperplane $H$ such that $|H\cap A|\geq \frac{1}{s-1}|A|$,
thus containing a significant fraction of the elements of $A$. In
fact, this corollary is sometimes given as the statement of the
Freiman $2^d$--Theorem itself, in part because it can be shown to
easily imply the version given above, illustrating the close dual
relationship between being covered by a small number of hyperplanes
and having a large intersection with a hyperplane.

The Freiman $2^d$--Theorem was one of the main tools used in the
original proof of Freiman's Theorem \cite{bilu} \cite{F2} \cite{F1}
(a result which shows that any subset $A\subseteq \Z$ with
$|A+A|\leq C|A|$ must be a large subset of a multidimensional
progression), which has become one of the foundational centerpieces
in inverse additive theory. However, like Freiman's Theorem itself,
it suffers from lacking even asymptotically correct constants.
Remedying such a drawback would greatly magnify the applicability of
these results, and in the case of Freiman's Theorem, much effort has
been so invested culminating in the achievement of values that are
now almost asymptotically correct \cite{Ch}.

With the Freiman $2^d$--Theorem, there has been less notable success
in improving the constants. When $d=2$ (so that a hyperplane is just
a line), independent proofs of the result were found by Fishburn
\cite{Fi} and by Stanchescu \cite{S2}, with the latter method
yielding an optimal value for $s(c,d)$ (specifically, $s=s(c,2)$ is
the ceiling of the smaller root defined by
$c|A|=4|A|+1-2(s+\frac{|A|}{s})$), though the value for $k(c,d)$ was
still not asymptotically accurate (the constant obtained was cubic
in $s$ rather than quadratic).

The main result of this paper is the following, which extends the
$2$-dimensional case of the Freiman $2^d$--Theorem to distinct sets
while at the same time giving exact values for the constants (when
$||A|-|B||\leq s$).

\begin{thm}\label{final-result} Let $s\geq 2$ be an integer, and let
$A,\,B\subseteq \mathbb{R}^2$ be finite subsets.

{\rm (i)} If $||A|-|B||\leq s$, $|A|+|B|\geq 4s^2-6s+3$, and
\be\label{sim-assumption}|A+B|<(2-\frac 1 s )(|A|+|B|)-2s+1,\ee then
$h_1(A,B)< s$.

{\rm (ii)} If $|A|\geq |B|+s$, $|B|\geq
2s^2-\frac{7}{2}s+\frac{3}{2}$, and
\be\label{speel}|A+B|<|A|+(3-\frac 2 s )|B|-s,\ee then $h_1(A,B)<s$.
\end{thm}

The following example shows that, for $s\geq 3$, the constant in (i)
is best possible: let $T$ be a right isosceles triangle in the
integer lattice whose equal length sides each cover $x=2s-2$ lattice
points; then $|T|=(s-1)(2s-1)$ and
$|2T|=2(s-1)(4s-5)<4|T|+1-2s-2\frac{|T|}{s}$, but $T$ is covered by
no fewer than $2s-2>s-1$ parallel lines. The same example shows that,
even when $|A|+|B|<4s^2-6s+3$ and $h_1(A,B)\geq s$, the lower bound
on $|A+B|$ implied by Theorem \ref{final-result} (i) is quite
accurate.  Indeed, when $x\geq s$, we have $|T|=\frac{x(x+1)}{2}\geq
\frac{s(s+1)}{2}$, $h_1(T,T)\geq s$ and
$$|2T|=x(2x-1)=4|T|+\frac{3}{2}-3\sqrt{\frac{1}{4}+2|T|}.$$ On the
other hand,  for $|A|+|B|<4s^2-6s+3$ and $h_1(A,B)\geq s$, one can
always choose $s_0<s$ so that the hypothesis of Theorem
\ref{final-result} hold. Let
$t_0=\frac{1}{2}\sqrt{\frac{1}{4}+|A|+|B|}-\frac{1}{4}$, and let
$s_0=\lceil t_0\rceil=t_0+z$, with $0\leq z<1$. Note that $|A|+|B|=
4(t_0+1)^2-6(t_0+1)+2>4s_0^2-6s_0+2$. When $|A|+|B|\geq 14$, by
applying Theorem \ref{final-result} with $s_0$,   the resulting
bound, as a function of $z$, is minimized for $z=0$. Consequently,
we obtain the estimate
$$|A+B|\geq 2|A|+2|B|+\frac{1}{2}-3\sqrt{\frac{1}{4}+|A|+|B|}$$ when $14\leq |A|+|B|\leq
4s^2+2s$, $h_1(A,B)\geq s$, and either $||A|-|B||\leq s_0$ or else
$||A|-|B||\leq \lceil\frac{s}{2}\rceil$ and $s(s+1)\leq |A|+|B|$.
This shows that the resulting bound for $|A+B|$ using $s_0$ is
surprisingly accurate for $|A|+|B|\geq s(s+1)$. However, once
$|A|+|B|<s(s+1)$, the lower bound for $|A+B|$ assuming $h_1(A,B)\geq
s$ should begin to become much larger.

The proof of Theorem \ref{final-result} will be given in Section 3,
along with the proof of the dual formulation bounding $|A+B|$ when
$A$ and $B$ are assumed to contain no $s$ collinear points.
Concerning the case $s=2$, a result of  Ruzsa \cite{R}, generalizing
to distinct sets yet another result of Freiman \cite[Eq. 1.14.1]{F1}
\cite{taobook}, shows that if $A,\,B\subseteq \R^d$ with $|A|\geq
|B|$ and $A+B$ $d$-dimensional, then $|A+B|\geq
|A|+d|B|-\frac{d(d+1)}{2}$. However, as the Freiman $2^d$--Theorem
indicates, the cardinality of $A$ and $B$ modulo appropriate
subspaces also plays an important role contributing to the
cardinality of $A+B$. Section 2 is devoted to proving Theorem
\ref{THEbounds} below, which gives a general lower bound for $|A+B|$
based upon $|\phi_H(A)|$ and $|\phi_H(B)|$, with $H=\R x_1$ an
arbitrary one-dimensional subspace. It will be a key ingredient in
the proof of Theorem \ref{final-result}. We remark that the
symmetric case (when $A=B$) was first proved by Freiman \cite[Eq.
1.15.4]{F1}.

\begin{thm}\label{THEbounds} Let $A,\,B\subseteq \mathbb{R}^2$ be finite, nonempty
subsets, let $\ell=\R x_1$ be a line, let $m$ be the number of lines
parallel to $\ell$ which intersect $A$, and let $n$ be the number of
lines parallel to $\ell$ that intersect $B$. Then
\be\label{la-lb-bound}|A+B|\geq
(\frac{|A|}{m}+\frac{|B|}{n}-1)(m+n-1).\ee Furthermore, the
following bounds are implied by (\ref{la-lb-bound}).

{\rm (i)} If $m\geq n$ and $|A|\leq |B|+m$, then $$|A+B|\geq
(2-\frac 1 m)(|A|+|B|)-2m+1.$$

{\rm (ii)} If $|A|\geq |B|+m$, then $$|A+B|\geq |A|+(3-\frac 2
m)|B|-m.$$

{\rm (iii)} If $1<m<|A|$, let $l$ be an integer such that
$\frac{l(l-1)}{m(m-1)}\leq \frac{|B|}{|A|-m}\leq
\frac{l(l+1)}{m(m-1)}$, and if $m=1$, let $l=1$. Then
$$|A+B|\geq
|A|+|B|+\frac{l-1}{m}|A|+\frac{m-1}{l}|B|-(m+l-1).$$

{\rm (iv)} In general, $$|A+B|\geq
|A|+|B|+2\sqrt{(m-1)(\frac{|A|}{m}-1)|B|}-(\frac{|A|}{m}+m)+1.$$
\end{thm}

Note $l=\lfloor
\sqrt{\frac{1}{4}+\frac{(m-1)|B|}{|A|/m-1}}+\frac{1}{2}\rfloor$
satisfies the hypotheses of Theorem \ref{THEbounds}(iii) for
$m<|A|$. We remark that Theorem \ref{THEbounds}(iv), along with the
compression techniques of Section 2, easily implies (a diagonal
compression along $x_1-x_2$ should also be used when $A$ is
contained in two lines, $y_1+\R x_1$ and $y_2+\R x_2$, each
containing $\frac{|A|+1}{2}$ points of $A$) the $2$-dimensional case
of a discrete analog of the Brunn-Minkowski Theorem given by Gardner
and Gronchi \cite[Theorem 6.6, roles of $A$ and $B$ reversed]{Ga}.
%with equality between the two bounds
%possible only when $h_1(A,A)=2$ or $A$ contains no $3$ collinear
%points (though this latter case might be possibility is likely not
%possible).
Also, (\ref{la-lb-bound}) improves the $2$-dimensional case of an
estimate of Green and Tao \cite[Theorem 2.1]{GT}, with the two
bounds equal only when $A$ is a rectangle. In Section 2.1, we give a
continuous version of Theorem \ref{THEbounds} that generalizes the
$2$-dimensional case of the Brunn-Minkowski Theorem (see e.g.
\cite{Ga}).

The lower bounds for $|A+B|$ from Theorem \ref{final-result}(ii) and Theorem \ref{THEbounds}(ii) are
estimates based on $\min\{|A|,\,|B|\}$, much like nearly all other
existing estimates for distinct sumsets; however, if $|A|$ is much
larger than $|B|$, such bounds can be weak. The bounds in Theorem \ref{THEbounds}(iii) and Theorem \ref{THEbounds}(iv) are more accurate since they take into account the relative size of $|A|$ and $|B|$. It would be desirable to have a similar refinement to Theorem \ref{final-result}, i.e., a lower bound for $|A+B|$ based off the parameter $s\leq h_1(A,B)$ \emph{and} the relative size of $|A|$ and $|B|$. One possibility would be if the bound in Theorem \ref{THEbounds}(iii) held with the globally defined parameter $s\leq h_1(A,B)$ in place of $m$, for $|A|$ and $|B|$ suitably large with respect to $s$. This is achieved by Theorem \ref{final-result}(i) for the extremal case when $|A|$ and $|B|$ are very close in size. Theorem \ref{Thm-AisBig} below accomplishes the same aim for the other extremal case, when $|A|$ is much larger than $|B|$.  Note that the coefficient of $|B|$ in the bound below is much larger
than the value of $3-\frac{2}{s}$ obtained from Theorem
\ref{final-result}(ii). Moreover, the bound on $|B|$ required to apply Theorem \ref{Thm-AisBig}(b) is much smaller than the corresponding requirement for Theorem \ref{final-result}, being linear in $s$ rather than quadratic. In fact, Theorem \ref{Thm-AisBig}(a) shows that, by only increasing slightly the requirement of $|A|$ to be much larger than $|B|$---from $|A|\geq \frac{1}{2}s(s-1)|B|+s$ to $|A|>
\frac{1}{8}(2s-1)^2|B|-\frac{1}{4}(2s-1)+\frac{(s-1)^2}{2(|B|-2)}$---one can eliminate all need for $|A|$ and $|B|$ to be sufficiently large with respect to $s$.

\begin{thm}\label{Thm-AisBig} Let $s$ be a positive integer, and
let $A,\,B\subseteq \R^2$ be finite, nonempty subsets with
$h_1(A,B)\geq s$ and $|A|\geq \frac{1}{2}s(s-1)|B|+s$. If either

(a) $|A|>
\frac{1}{8}(2s-1)^2|B|-\frac{1}{4}(2s-1)+\frac{(s-1)^2}{2(|B|-2)}$,
or

(b) $|B|\geq \frac{2s+4}{3}$, then \be\label{Goool-Aisbig}|A+B|\geq
|A|+s(|B|-1).\ee
\end{thm}

We remark that the bound $|A|\geq \frac{1}{2}s(s-1)|B|+s$ is
not in general sufficient to guarantee $|A+B|\geq
|A|+s(|B|-1)$, and thus the slight increase in the requirement for $|A|$ given by (a) is necessary. For instance, let $s=34$, and let $A'$
and $B$ be geometrically similar right isosceles triangles whose
equal length sides each cover $82$ and $3$ lattice points,
respectively. Suppose $A'$ lies in the positive upper plane with one
its equal length sides along the horizontal axis. Let $A$ be
obtained from $A'$ by deleting the $3$ points in $A'$ farthest away
from the horizontal axis. Then $|B|=6$,
$|A|=3400=\frac{1}{2}s(s-1)|B|+s$, $h_1(A,B)=80>34$, and
$|A+B|=3567<3570=|A|+s(|B|-1)$. As a second example, let
$A=[0,a-1]\times [0,s+1]$ and $B=[0,b-1]\times \{0,1\}$ be two
rectangles in the integer lattice. We have $|A|=a(s+2)$, $|B|=2b$
and $|A+B|=(a+b-1)(s+3)=|A|+s(|B|-1)+a-b(s-3)-3$. By taking
$b=(s+3)/6$ and  $a=(s(s-1)b+s+1)/(s+2)=(s^2+3)/6$   (with $s\equiv
3 \pmod 6$), we have $|A|=\frac{1}{2}s(s-1)|B|+s+1$, $|B|=(s+3)/3$
and $|A+B|<|A|+s(|B|-1)$. Furthermore, $h_1(A,B)\geq
h_1(A,A)\geq \min\{s+2,\,(s^2+3)/6\}\geq s$ for $s\ge 9$.

We conclude the introduction with two special cases of Freiman's
Theorem for which exact constants are known. The first is folklore
\cite{Natbook} \cite{taobook}, while the second is a generalization
by Lev and Smeliansky \cite{LS} of the Freiman $(3k-4)$--Theorem
\cite[Theorem 1.9]{F1} \cite{Natbook} \cite{taobook}.

\begin{theirtheorem}\label{CDT-for-Z} If $A$ and $B$ are finite and
nonempty subsets of a torsion-free abelian group, then
\be\label{zvirtue1}|A+B|\geq |A|+|B|-1,\ee with equality possible
only when $A$ and $B$ are arithmetic progressions with common
difference or when $\min\{|A|,\,|B|\}=1$.
\end{theirtheorem}

\begin{theirtheorem}\label{thm-Freiman-3k-3-stanchescu_version} Let
$A,\,B\subseteq \Z$ be finite nonempty subsets with $0=\min A=\min
B$, $\max A\geq \max B$ and $\gcd(A)=1$. Let $\delta=1$ if $\max
A=\max B$, and let $\delta =0$ otherwise. If
$$|A+B|=|A|+|B|+r\leq |A|+2|B|-3-\delta,$$  then $\max A\leq
|A|+ r$.
\end{theirtheorem}

\section{Lower Bound Estimates via Compression}

\subsection{Discrete Sets}

\indent \indent Let $X=(x_1,x_2,\ldots,x_d)$ be an ordered basis for
$\mathbb{R}^d$,  and let $X_i=\langle x_1,\ldots,x_i\rangle$ for
$i=0,\ldots,d$. Let $A \subseteq \mathbb{R}^d$ be a finite subset.
The linear compression of $A$ with respect to $x_i\in X$, denoted
$\mathbf{C}_i (A)=\mathbf{C}_{X,i}(A)$, is the set obtained by
compressing and shifting $A$ along each line $\R x_i+a$, where $a\in
\R^d$, until the resulting set $\C_i(A)\cap (\R x_i+a)$ is an
arithmetic progression with difference $x_i$ whose first term is
contained in the hyperplane $H=\langle
x_1,\ldots,x_{i-1},x_{i+1},\ldots,x_d\rangle$. More concretely, we
define the set $\C_i(A)$ piecewise by its intersections with the
lines $(\R x_i+a)$, $ a\in \R^d$,    by letting $\C_i(A)\cap (\R
x_i+a)$ be the subset of $\R x_i+a$ satisfying
$$\phi_{H}(\C_i(A)\cap (\R x_i+a))=\{0,x_i,2x_i,\ldots,(r-1)x_i\},$$
where $r=|A\cap (\R x_i+a)|$ and the right hand side is considered
empty if $r=0$. We let
$$\C_X(A)=\C_d(\C_{d-1}\ldots(\C_1(A)))$$
be the fully compressed subset obtained by iteratively compressing
$A$ in all $d$ dimensions. Observe that
\be\label{compression-doesn't-affect-param}|\phi_{X_i}(\C_X(A))|=|\phi_{X_i}(A)|,\ee
for $i=0,\ldots,d$.

Compression techniques in the study of sumsets have been used by
various authors, including Freiman \cite{F1}, Kleitman \cite{K},
Bollob\'as and Leader \cite{BL}, and Green and Tao \cite{GT}. The
reason for introducing the notion of compression is that it gives a
useful lower bound for the sumset of an arbitrary pair of finite
subsets $A,\,B\subseteq \R^d$. Namely, letting $H$ be as above and
letting $C_t$ denote $C\cap (\R x_i+t)$ below, we have in view of
Theorem \ref{CDT-for-Z} that
 \begin{eqnarray}\label{well-align-lowerbound}
 |A+B|&=&\sum_{t\in H} |(A+B)_t|\nonumber \\
    &\ge& \sum_{t\in H} \max\{ |A_s+B_{t-s}|:\; A_s\neq\emptyset, B_{t-s}\neq\emptyset\}\nonumber \\
    &\ge&\sum_{t\in H} \max\{ |A_s|+|B_{t-s}|-1:\; A_s\neq\emptyset, B_{t-s}\neq\emptyset\}\nonumber\\
     &=&   |\C_i(A)+\C_i (B)|,
        \end{eqnarray} and
consequently (by iterative application of
(\ref{well-align-lowerbound})),
\be\label{well-align-lowerbound-fullycompressed} |A+B|\geq
|\C_X(A)+\C_X(B)|.\ee

We now restrict our attention to the case $d=2$, which is the object
of study for this paper. Let $m=|\phi_{X_1}(A)|$,
$n=|\phi_{X_1}(B)|$, $A_i=\C_X(A)\cap (\R x_1+(i-1)x_2)$ and
$B_i=\C_X(B)\cap (\R x_1+(i-1)x_2)$. Note that $|A_1|\geq |A_2|\geq
\ldots \geq |A_{m}|$ and $|B_1|\geq |B_2|\geq \ldots \geq |B_{n}|$. If
$|A_i|=a_i$ and $|B_j|=b_j$, then \be\label{calc-form-wellaligned}
|\C_X(A)+\C_X(B)|=\Sum{l=2}{m+n} \underset{i}{\max}\{a_i+b_{l-i}\mid
1\leq i\leq m,\,1\leq l-i\leq n\}-(m+n-1).\ee Consequently, the
following lemma provides a lower bound for $|A+B|$ based upon the
number of parallel lines that cover $A$ and $B$, which will imply
(\ref{la-lb-bound}) in Theorem \ref{THEbounds}.

\begin{lem}\label{red-bound-calculus-formulation} If $a_1,\ldots ,a_m,b_1,\ldots ,b_n\in \R$,
then \be\label{2-red-bound}
\frac{1}{m+n-1}\Sum{i=2}{m+n}\underset{j}{\max} \{
a_j+b_{i-j}:\;1\le j\le m, 1\le i-j\le n\} \geq
\frac{1}{m}\Sum{i=1}{m}a_i+\frac{1}{n}\Sum{i=1}{n}b_i.\ee
\end{lem}

\begin{proof} The proof is by induction on $m+n$. The result clearly
holds if either $m=1$ or $n=1$. Assume that $m,\, n\ge 2$. Let $a =
(a_1,\ldots ,a_m)$ and $b =(b_1,\ldots ,b_n)$. For a vector
$x=(x_1,x_2,\ldots ,x_k)$, we denote by
$\overline{x}=\frac{1}{k}\Sum{i=1}{k}x_i$. Also, if $y=(y_1,\ldots
,y_l)$, we denote by
$$
u(x, y)=\Sum{i=2}{k+l}\underset{j}{\max} \{ x_j+y_{i-j}:\;1\le j\le
k, 1\le i-j\le l\}.
$$
Thus we want to prove
$$
u(a, b)\ge (m+n-1)(\bar{a}+\bar{b}).
$$
Let $a'=(a_2,\ldots ,a_{m})$ and $b'=(b_2,\ldots ,b_{n})$. We may
assume that $\bar{a}-\bar{a'}\le \bar{b}-\bar{b'}$. We clearly have
    $u(a,b)\ge u(a',b)+a_1+b_1.$
Thus by the induction hypothesis,
    \begin{eqnarray*}
    u(a,b)&\geq&(m+n-2)(\bar{a'}+\bar{b})+a_1+b_1\\
    &=&(m+n-2)(\bar{a'}+\bar{b})+m\bar{a}-(m-1)\bar{a'}+n\bar{b}-(n-1)\bar{b'}\\
    &=&(m+n-1)(\bar{a}+\bar{b})+(n-1)(\bar{a'}-\bar{a})+(n-1)(\bar{b}-\bar{b'}) \\
    &\ge&(m+n-1)(\bar{a}+\bar{b}),
    \end{eqnarray*}
as claimed.
\end{proof}

Note that taking $a_i=\frac{1}{m}\Sum{k=1}{m}a_k$ and
$b_j=\frac{1}{n}\Sum{k=1}{n}b_k$ for all $i$ and $j$ shows that
equality can hold in (\ref{2-red-bound}). More generally, equality
holds whenever $a_1,\ldots,a_{m}$ and $b_1,\ldots,b_{n}$ are
arithmetic progressions of common difference. We now prove Theorem
\ref{THEbounds}.

\begin{proof} {\it of Theorem \ref{THEbounds}.} The bound in (\ref{la-lb-bound}) follows from
Lemma \ref{red-bound-calculus-formulation},
(\ref{calc-form-wellaligned}),
(\ref{well-align-lowerbound-fullycompressed}) and
(\ref{compression-doesn't-affect-param}). Consider the bound given
by (\ref{la-lb-bound}) as a discrete function in the variable $n$.
If $m=|A|$, then maximizing $n$ will minimize (\ref{la-lb-bound}).
Otherwise, it is a routine discrete calculus minimization question
to determine that $l=\lfloor
\sqrt{\frac{1}{4}+\frac{(m-1)|B|}{|A|/m-1}}+\frac{1}{2}\rfloor$ is
the value of $n$ which minimizes (\ref{la-lb-bound}), and that $l-1$
also minimizes the bound when
$\sqrt{\frac{1}{4}+\frac{(m-1)|B|}{|A|/m-1}}+\frac{1}{2}\in \Z$.
Rearranging the expression for $l$ yields (iii). If $m\geq n$ and
$|A|\leq |B|+m$, then $l\geq m\geq n$ follows, whence the minimum of
(\ref{la-lb-bound}) occurs instead at the boundary value $n=m$,
yielding (i). If $|A|\geq |B|+m$, then (\ref{la-lb-bound}) implies
that
$$|A+B|\geq |A|+|B|+\frac{n-1}{m}(|B|+m)+\frac{m-1}{n}|B|-(m+n-1).$$
Considering the left hand side as a discrete function in $n$, it is
another routine discrete calculus computation to determine $n=m$
minimizes the bound. This yields (ii). Note that when $|B|=|A|+m$
the bounds in (ii) and (i) are equal. Finally, considering the bound
given by (\ref{la-lb-bound}) as a continuous function in $n$, it
follows that $n=\sqrt{\frac{(m-1)|B|}{|A|/m-1}}$ minimizes the bound
in (\ref{la-lb-bound}) when $|A|>m$. This yields (iv) except in the
case $|A|=m$, in which case the trivial bound $|A+B|\geq |B|$
implies (iv) instead.
\end{proof}

\subsection{Measurable Sets}

\indent\indent Let $\mu_d$ be the Lebesgue measure on the space
$\R^d$, $d\geq 1$, and let $\{x_1,\ldots,x_d\}$ be the $d$ standard
unit coordinate vectors for $\R^d$. In this subsection, we briefly
show how the results of the previous section are related to sumset
volume estimates, such as the Brunn-Minkowski Theorem
\cite{taobook,Ga}. In what follows, we make implicit use of the
basic analytic theory regarding the Lebesgue measure (see e.g.
\cite{rudin}).

\begin{theirtheorem}[Brunn-Minkowski
Theorem]\label{brunn-minkowski} If $A,\,B\subseteq \R^d$ and $A+B$
are nonempty, measurable subsets, then \be\label{brunn-mink-bound}
\mu_d(A+B)^{1/d}\geq \mu_d(A)^{1/d}+\mu_d(B)^{1/d}.\ee
\end{theirtheorem}

Let $\phi_i:\R^2\rightarrow \R$ denote the canonical projection onto
the $i$-th coordinate, $i=1,2$.  Theorem \ref{brunn-mink-case} below
can be regarded as an extension of Theorem \ref{THEbounds} to the
continuous case. Since there are measurable sets $X\subset \R^2$
with $\phi_1 (X)$ not $\mu_1$--measurable, the assumption of $\phi_1
(A)$ and $\phi_1 (B)$ being measurable in Theorem
\ref{brunn-mink-case} is necessary. However, without this condition,
one may always find subsets $A'\subset A$ and $B'\subset B$ with
$\mu_2 (A\setminus A')=\mu_2 (B\setminus B')=0$ such that $\phi_1
(A')$, $\phi_1 (B')$ and $A'+B'$ are measurable (this will be
evident from the proof). Thus, Theorem \ref{brunn-mink-case} implies
the $2$-dimensional Brunn-Minkowski bound, with equality between the
two bounds only possible when
$$\mu_1(\phi_1(A'))\sqrt{\mu_2(B)}=\mu_1(\phi_1(B'))\sqrt{\mu_2(A)}.$$
 The  condition
$0<\mu_1(\phi_1(A')),\,\mu_1(\phi_1(B'))<\infty$ is not highly
restrictive since $\mu_1(\phi_1(A'))=0$ implies $\mu_2(A)=0$, and if
$\mu_1(\phi_1(A'))=\infty$, then either $\mu_2(A+B)=\infty$ or
$\mu_2(B)=0$. Thus the condition could be omitted if all indefinite
expressions were interpreted to equal zero.

\begin{thm}\label{brunn-mink-case} If $A,\,B\subseteq \R^2$, $\phi_1 (A)$, $\phi_1 (B)$ and $A+B$ are nonempty
measurable subsets with
$0<\mu_1(\phi_1(A)),\,\mu_1(\phi_1(B))<\infty$, then
\be\label{cont-case-bound}\mu_2(A+B)\geq
\left(\frac{\mu_2(A)}{\mu_1(\phi_1(A))}+\frac{\mu_2(B)}{\mu_1(\phi_1(B))}\right)(\mu_1(\phi_1(A))+\mu_1(\phi_1(B))).\ee
\end{thm}

\begin{proof} The theory of compressions can be extended
to include measurable subsets of $\R^d$, though some care is needed
to verify all the basic properties still hold. For simplicity, we
restrict our attention to the case $d=2$.  Due to the extra care
that needs to be taken concerning nullsets and the measurability of
various sets, we have included many more details than would
otherwise be necessary. We may assume that $\mu_2(A+B)$ is finite,
and thus $\mu_2(A)$ and $\mu_2(B)$ as well, else the theorem is
either trivial or meaningless.

For a subset $X\subseteq \R^2$  and $i\in \{1,2\}$, let $f_{X,i}:
\phi_{3-i}(X)\rightarrow [0,\infty]$ be defined as
$f_{X,i}(\phi_{3-i}(x))=\mu_1(X\cap (\R x_i+x))$ if $X\cap (\R x_i+x)$
is measurable and otherwise $f_{X,i}(\phi_{3-i}(x))=0$. We define the
linear compression $\C_i(X)$, for $i=1,2$, by it intersections with
the lines $(\R x_i+a)$, $a\in \R^2$, by letting $\C_i(X)\cap (\R
x_i+a)$ be the subset of $\R x_i+a$ defined by
$$\phi_{i}(\C_i(X)\cap (\R x_i+a))=[0,f_{X,i}(\phi_{3-i}(a))],$$
if $X\cap (\R x_i+a)$ is nonempty, and letting $\C_i(X)\cap (\R
x_i+a)$ be empty otherwise. Let
$$E_i(X):=\{x\in \C_i(X)\mid \phi_{i}(x)=f_{X,i}(\phi_{3-i}(x))\}$$ be those
points with maximal $x_i$ coordinate in $\C_i(X)$.

We recall that an arbitrary measurable subset $A\subseteq \R^2$
contains an $F_{\sigma}$--set $A'$ with $\mu_2(A\setminus A')=0$.
 By the
continuity of addition, the sumset of two   $F_{\sigma}$--sets is an
$F_{\sigma}$--set, and thus measurable. Similarly, the projection
$\phi_1 (A')$ is also an $F_{\sigma}$-set and thus $\mu_1$-measurable.

Suppose now that $\phi_1 (A)$ is measurable. Then
$U=\phi_1(A)\setminus \phi_1 (A')$ is also measurable. Let
$U'\subset U$ be an $F_{\sigma}$--set with $\mu_1 (U\setminus
U')=0$. Then $\tilde{A}=A'\cup (\phi_1^{-1}(U')\cap \R x_1)$ is also
an $F_{\sigma}$--set with $\mu_2 (\tilde{A})=\mu_2(A)$ and
$\mu_1(\phi_1(\tilde{A}))=\mu_1(\phi_1 (A))$.

Since each closed subset can be written as a countable union of
compact subsets, we have $\tilde{A}=\bigcup_{i=1}^{\infty}F_i$ with
$F_1\subseteq F_2\subseteq \ldots$ and each $F_i$ a compact subset.
Furthermore, each $F_i=\bigcap_{j=1}^{\infty}S^i_j$, with
$S^i_1\supseteq S^i_2\supseteq \ldots$ and each $S^i_j$ a finite
union of cubes (a cartesian product of   closed intervals). Passing
through cubes and compact sets, it follows that any section
$\tilde{A}\cap (\R x_i+a)$ of an $F_{\sigma}$--set is also an
$F_{\sigma}$--set (with respect to $\mu_1$). By the upper continuity
of $\mu_1$, we have
$\C_k(\tilde{A})=\C_k(\bigcup_{i=1}^{\infty}F_i)=\bigcup_{i=1}^{\infty}\C_k(F_i)\cup
\tilde{A}_1$, for $k=1,2$, where $\tilde{A}_1$ is a disjoint subset
contained in $E_k(A')$ and
$\phi_{3-k}(\C_k(\tilde{A}))=\phi_{3-k}(\bigcup_{i=1}^{\infty}\C_k(F_i))$.
On the other hand, since each compact set $F_i$ is bounded, then the
lower continuity of $\mu_1$ implies
$\C_k(\bigcap_{j=1}^{\infty}S^i_j)=\bigcap_{j=1}^{\infty}\C_k(S^i_j)$,
for $k=1,2$. Note that $\C_k(S^i_j)$, for $k=1,2$, is still a finite
union of cubes. Consequently, $\C_k(\tilde{A})\setminus \tilde{A}_1$
is an $F_{\sigma}$--set. We call
$\C(A)=\C_1(\C_2(\tilde{A})\setminus \tilde{A}_1)$ the {\it
compression} of $A$. We have
\ber\mu_1(\phi_1(A))=\mu_1(\phi_1(\tilde{A}))=\mu_1(\phi_1(\C_2(\tilde{A})))=
\mu_1(\phi_1(\C_2(\tilde{A})\setminus
\tilde{A}_1))=\mu_1(\phi_1(\C(A)). \label{grip4}\eer

Likewise define $\tilde{B}$, $\tilde{B}_1$ and $\C (B)$, and note that the corresponding equality in (\ref{grip4}) holds for $\C(B)$ as well.

Since  $\mu_2 (A+B)\ge \mu_1 (\tilde{A}\cap (\R x_2+a)) \mu_1(\phi_1
(B))$ for each $a\in \R^2$, then $\mu_1(\phi_1(B))>0$ and
$\mu_2(A+B)<\infty$ imply $\sup\{f_{\tilde{A}, 2}(x)\mid x\in
\phi_1(\tilde{A})\}<\infty$. Likewise for $\tilde{B}$.

Let $S_z=(\C_2(\tilde{A})\setminus \tilde{A}_1)\cap (\R x_1+z)$ be
an $x_1$--section. Observe that, if $\phi_2 (z)\le \phi_2 (z')$ then
$S_{z'}\subseteq S_z$ and thus $\mu_1(S_{z'})\leq \mu_1(S_z)$.
Consequently, $\C (A)$ consists precisely in the area between the
graph of the monotonic decreasing $L^+$--function
$f_{\C_2(\tilde{A})\setminus \tilde{A}_1,1}:[0,M)\rightarrow
[0,\mu_1(\phi_1(A))]$ and the $x_2$-axis,  where
$M=\sup\{f_{\tilde{A}, 2}(x)\mid x\in \phi_1(\tilde{A})\}$ (the
interval of domain may possibly be closed $[0,M]$ as well).  As both
$\mu_1(\phi_1(A))$ and $M$ are finite, $\C (A)$ is Riemann
integrable, and thus also measurable. The same is true for $\C (B)$,
from which it is then easily observed that their sumset $\C(A)+\C
(B)$ also consists of the area between the graph of a monotonic
decreasing $L^+$--function and the $x_2$-axis, and hence is measurable.

As $\C (A)$, $\tilde{A}$ and $\C_2(\tilde{A})\setminus \tilde{A}_1$ are
measurable, by Fubini's Theorem we have
\begin{eqnarray}
\mu_2(\C(A))&=&\int\!\!\!\!\int\chi_{\C_1(\C_2(\tilde{A})\setminus
\tilde{A}_1)}dx_1dx_2=\int\!\!\!\!\int\chi_{\C_2(\tilde{A})\setminus
\tilde{A}_1}dx_1dx_2=\mu_2(\C_2(\tilde{A})\setminus
\tilde{A}_1)\label{grip2} \\&=& \int\!\!\!\!\int\chi_{\C_2(\tilde{A})\setminus
\tilde{A}_1}dx_2dx_1=
\int\!\!\!\!\int\chi_{\C_2(\tilde{A})}dx_2dx_1=
\int\!\!\!\!\int\chi_{\tilde{A}}dx_2dx_1=\mu_2(\tilde{A})=\mu_2(A),\nn\end{eqnarray}
where $\chi_T$ denotes the characteristic function of the set $T$.
Likewise, \be \mu_2(\C (B))=\mu_2(B).\label{grip3}\ee

Since $\tilde{A}$ and $\tilde{B}$ are $F_{\sigma}$-sets, each $x_2$-section of $\tilde{A}$ or $\tilde{B}$ is also an
$F_{\sigma}$--set (with respect to $\mu_1$). Hence, letting $X_z$ denote in (\ref{hitman})
below the $x_2$-section $(\R x_2+z)\cap X$ of  $X\subseteq \R^2$,
\ber\nn\mu_1((A+B)_z)=\mu_1(\bigcup_{x+y=z}(A_x+B_y))\geq
\sup\{\mu_1(\tilde{A}_x+\tilde{B}_y)\mid x+y=z\}\\\geq\label{hitman}
\sup\{\mu_1(\tilde{A}_x)+\mu_1(\tilde{B}_y)\mid
x+y=z\}=\mu_1((\C_2(\tilde{A})+\C_2(\tilde{B}))_z),\eer for $z\in
\R^2$ such that $(A+B)_z$ is $\mu_1$-measurable, where the second inequality follows from the inequality
$\mu_1(X+Y)\geq \mu_1(X)+\mu_1(Y)$ (which is the case $d=1$ in the
Brunn-Minkowski Theorem). Using Fubini's Theorem and (\ref{hitman})
(for the first inequality; the second one follows by an analogous
argument), we infer
\begin{eqnarray}%\label{we1}
\mu_2(A+B)&=&\int\!\!\!\!\int\chi_{A+B}dx_2dx_1\geq
\int\!\!\!\!\int\chi_{\C_2(\tilde{A})+\C_2(\tilde{B})}dx_2dx_1
\nn\\
&=& \int\!\!\!\!\int\chi_{(\C_2(\tilde{A})\setminus
\tilde{A}_1)+(\C_2(\tilde{B})\setminus
\tilde{B}_1)}dx_2dx_1=\mu_2((\C_2(\tilde{A})\setminus
\tilde{A}_1)+(\C_2(\tilde{B})\setminus \tilde{B}_1))
\nn\\
&=&\int\!\!\!\!\int\chi_{(\C_2(\tilde{A})\setminus
\tilde{A}_1)+(\C_2(\tilde{B})\setminus \tilde{B}_1)}dx_1dx_2\geq
\int\!\!\!\!\int\chi_{\C_1(\C_2(\tilde{A})\setminus
\tilde{A}_1)+\C_1(\C_2(\tilde{B})\setminus
\tilde{B}_1)}dx_1dx_2%\label{we2}
\nn\\
&=&\mu_2(\C (A) +\C (B))\label{grip1}.\eer

In view of (\ref{grip1}), (\ref{grip2}), (\ref{grip3}) and
(\ref{grip4}), we see that it suffices to prove the theorem for
$A=\C (A)$ and $B=\C (B)$. Since these are Riemann integrable, and
thus can be approximated by rectangular strips of fixed height
$\log_{2^{n}}(\mu_1(\phi_2(A)))$ and
$\log_{2^{n}}(\mu_1(\phi_2(B)))$ when $n\rightarrow \infty$, it thus
suffices to prove the theorem for unions of $2^n$ rectangular strips
of equal height, $n\in \Z^+$. We proceed by induction. If $n=1$, so
that both $A$ and $B$ are themselves rectangles of width
$\mu_1(\phi_1(A))$ and $\mu_1(\phi_1(B))$ and height
$\frac{\mu_2(A)}{\mu_1(\phi_1(A))}$ and
$\frac{\mu_2(B)}{\mu_1(\phi_1(B))}$, respectively, then
(\ref{cont-case-bound}) follows trivially. So we assume $n>1$.
Translate $A$ and $B$ so that the $x_2$-axis passes through the
midpoints of $\phi_1(A)$ and $\phi_1(B)$, and let $A^+\subseteq A$
and
 $B^+\subseteq B$ be those points with nonnegative $x_1$-coordinate, and let
$A^-\subseteq A$ and
 $B^-\subseteq B$ be those with non-positive $x_1$-coordinate.
Observing that $\mu_2(A+B)\geq \mu_2(A^++B^+)+\mu_2(A^-+B^-)$ and
applying the induction hypothesis to each of $A^++B^+$ and $A^-+B^-$
yields (\ref{cont-case-bound}), completing the proof.
\end{proof}

\section{Two-Dimensional Sets}

\indent\indent Recall that $h_1(A,B)$ denotes the minimal positive
integer $s$ such that there exist $2s$ (not necessarily distinct)
parallel lines $\ell_1,\ldots,\ell_{s},\ell'_1,\ldots,\ell'_{s}$
with $A\subseteq \bigcup_{i=1}^{s}\ell_i$ and
$B\subseteq\bigcup_{i=1}^{s}\ell'_i$. The next Lemma is analogous to
\cite[Lemma 2.2]{S2} and provides an inductive step in the proof of
Theorem \ref{final-result}.

\begin{lem} \label{lem-stan-reduction}
Let $s\geq 3$ be an integer, and let $A,\,B\subseteq \mathbb{R}^2$
be finite subsets, with $|A|\geq |B|\ge s$, such that there are no
$s$ collinear points in either $A$ or $B$. Then either:

(a) $h_1(A,B)\leq 2s-3$, or

(b) there exist $a,\,b\in \mathbb{R}^2$, a line $\ell$,  a nonempty
subset $A_0\subseteq A$ and a subset $B_0\subseteq B$, such that
$A_0\subseteq a+\ell$, $B_0\subseteq b+\ell$, $|B_0|\leq |A_0|\leq
s-1$, and
    \be\label{eq-deletion-gives-2por2}
    |A'+B'|\leq |A+B|-2(|A_0|+|B_0|),
    \ee
where $A'=A\setminus A_0$ and $B'=B\setminus B_0$.
\end{lem}

\begin{proof} Let $\conv (X)$ denote the boundary of the convex
hull of $X$. Note, since $|A|\geq |B|\ge s$ and since neither $A$
nor $B$ contains $s$ collinear points, that both $A$ and $B$ must be
$2$-dimensional. We assume (b) is false and proceed to show (a)
holds. Note Claim \ref{claim:lat} below implies that $A$ and $B$ are
also contained in a translate of the lattice generated by $a_1-a_0$
and $a'_1-a_0$, though the particular translate may vary from $A$ to
$B$ to $A+B$.

\begin{claim}\label{claim:lat} If $f$ and $f'$ are two consecutive edges of $\conv (A)$
incident at the vertex $a_0$, with $a_1,\,a_1'\in \conv (A)\cap A$
the closest elements to $a_0$ in each of the edges $f$ and $f'$,
respectively, then the sumset $A+B$ is contained in a translate of
the lattice generated by the two vectors $a_1-a_0$ and $a'_1-a_0$.
\end{claim}

\begin{proof} We use an argument by Ruzsa \cite{R}.
 Let $b_0$ be a vertex of $\conv (B)$ such that $A^*=A\setminus \{ a_0\}$ and
 $B^*=(B\setminus \{ b_0\})+(a_0-b_0)$
are contained in the same open half plane determined by some line
through $a_0$. We may w.l.o.g. assume that $a_0=b_0=(0,0)$ and that
both $A^*$ and $B^*$ are contained in the open half plane of points
with positive abscissa. Let $x\in A+B$, $x\neq (0,0)$, and consider
all the expressions of $x$ written as a sum of elements taken from
$(A+B)\setminus \{ (0,0)\}$. Since $A$ and $B$ are finite sets, and
since all points in $A^*$ and $B^*$ have positive abscissa, it
follows that the number of summands in any such expression is
bounded. Take one expression $x= x_1+ x_2+\cdots + x_k$ with a
maximum number of summands. If $ x_i\in A^*+B^*$ for some $i$, then
$x_i$ can be split into two summands, one  in $A^*$ and one in
$B^*$, contradicting the maximality of $k$. Therefore $x$ can be
written as a sum of elements in $C=(A+B)\setminus ((A^*+B^*)\cup
\{(0,0)\})$.

Since (b) does not hold, it follows that $|C|\le 2$. Hence all
elements in $A+B$ are contained in the lattice generated by the two
elements of $C$. Let $e$ and $e'$ be the two edges incident with
$b_0$. Note we may assume the convex hull of the two rays parallel
to $e$ and $e'$ with base point $b_0=(0,0)$ is contained in the
convex hull of two rays parallel to $f$ and $f'$ with base point
$a_0=(0,0)$, since otherwise by removing $a_0$ from $A$ we lose all
the points in either $|a_0+(B\cap e)|$ or $|a_0+(B\cap e')|$,
yielding (b). However, in this case, it is easily seen that
$\{a_1,a'_1\}\subseteq C$, whence $|C|=2$ implies $C=\{a_1,a'_1\}$,
completing the claim.
\end{proof}

\begin{claim}\label{claim:par} For each side $e$ of $\conv (B)$, there is a side
$f$ of $\conv (A)$, parallel to $e$, such that both $A-f+e$ and $B$
are contained in the same half plane defined by $e$. Moreover,
$|B\cap e|\le |A\cap f|$.
\end{claim}

\begin{proof} Let $\ell$ be the line parallel to
$e$ that intersects $A$, and for which $A-\ell+e$ and $B$ are both
contained in the same half plane defined by $e$. Let $f=\ell\cap
\conv(A)$ and let $A_{f}=A\cap \ell$.  In view of Theorem
\ref{CDT-for-Z}, we see that by removing the elements of $A_{f}$ we
lose $|A_{f}+B_{e}|\ge |A_{f}|+|B_{e}|-1$ elements from $A+B$, where
$B_e=B\cap e$. Since (b) does not hold, it follows that
$|A_{f}|+|B_e|-1< 2|A_{f}|$, whence $2\le |B_e|\le |A_{f}|$. In
particular, $f$ is an edge of the convex hull of $A$.
\end{proof}

Let $e$ and $e'$ be two consecutive edges of $\conv (B)$, and let
$f$ and $f'$ be the corresponding parallel edges in $\conv (A)$ as
given by Claim \ref{claim:par}. Denote the elements in $B_e:=B\cap
e$ by $b_0,\,b_1,\,\ldots,\,b_t$, ordered as they occur in the edge
$e$, and the ones in $A_f:=A\cap f$ by $a_0,\,a_1,\,\ldots ,\,a_r$,
ordered in the same direction as those of $B_e$. Likewise define
$b'_0=b_0,\,b_1',\,\ldots,\,b_{t'}'$ and $a'_0,\,a_1',\,\ldots
,\,a_{r'}'$ for the points in $B_{e'}:=B\cap e'$ and $A_{f'}:=A\cap
f'$. Note $a_0=a'_0$ need not hold, though as we will soon see
(Claim \ref{claim:edge}) this cannot fail by much.

\begin{claim}\label{claim:ap} With the notation above,
$b_0-b_1=a_0-a_1$.\end{claim}

\begin{proof} Let $f''\neq f$ be the edge adjacent to $a_0$ and let
$a''\neq a_1$ be the element of $\conv(A)\cap A$ adjacent to $a_0$.
If the claim is false, then, by removing $a_0$ from $A_{f}$ and
$b_0$ from $B_e$, we lose from $A+B$ the distinct elements
$a_0+b_0$, $a_0+b_1$, $a_1+b_0$ and either $b_0+a''$ or $a_0+b'_1$,
yielding (b).
\end{proof}

\begin{claim}\label{claim:edge} With the notation above, either: (i) $f$ and $f'$
are also consecutive, or (ii) they are separated by a single edge
$g$ of $\conv (A)$, and $A\cap g$ contains exactly two points.
\end{claim}

\begin{proof}
Traverse the convex hull of $A$, beginning at $a_0$ and in the
direction not given by $f$. Let $a_0,c_1,c_2,\ldots,c_k,a'_0$ be the
sequence of points on $\conv(A)$  encountered until the first point
$a'_0$ of $f'$ is reached. If the claim is false, then $k\geq 1$.
Hence, by removing $a_0$ from $A$ and $b_0$ from $B$, we lose from
$A+B$ the elements $a_0+b_0$, $b_0+a_1$, $b_0+c_i$ for
$i=1,\ldots,k$, and $b_0+a'_0$, yielding (b).
\end{proof}

Following our current notation, let $e''\neq e$ and $f''\neq f$ be
the edges of $\conv (B)$ and $\conv (A)$ incident to $b_t$ and
$a_r$, respectively. Denote by $a''_0=a_r,\, a''_1,\,\dots
,\,a''_{r''}$ and $b''_0=b_t,\, b''_1,\,\ldots ,\,b''_{t''}$ the
elements of $A_{f''}:=A\cap f''$ and $B_{e''}:=B\cap e''$, ordered
as they occur in their respective edge.

By an appropriate affine transformation, we may  assume that
$b_0=(0,0)$, $b_1=(1,0)$ and $b'_1=(0,1)$  and that both $A$ and $B$
are contained in the positive first quadrant. We denote by
$\phi_1:\R^2\rightarrow  \R$ the projection onto the first
coordinate. Let $A_i=A\cap \{y=i\}$ and let $B_i=B\cap \{y=i\}$.

If $\phi_1(b_t)> \phi_1(a_r)-\phi_1(a_0)$, and in particular, if
$\phi_1(b_t)> \phi_1(a_r)$, then the removal of $A_0$ from $A$
results in a loss of at least $|b_0+A_0|+|b_t+A_0|=2|A_0|$ elements
from $A+B$, yielding (b). Therefore,
\begin{equation}\label{eq:a0}\phi_1(b_t)\le \phi_1(a_r)-\phi_1(a_0).\end{equation}
Furthermore, if $\phi_1(b_t)= \phi_1(a_r)-\phi_1(a_0)$, then we
likewise conclude that (b) holds, by removing $A_0$ from $A$, unless
$A_0+B_0=\{b_0,b_t\}+A_0$. However, in view of Claims
\ref{claim:lat} and \ref{claim:ap}, this is only possible if $A_0$
is an arithmetic progression of difference $a_1-a_0$. We proceed in
two cases.

\textbf{Case A:} Claim \ref{claim:edge}(i) holds for the pair $f$
and $f'$. In this case, $a_0=a'_0$ and w.l.o.g. $a_0=b_0=(0,0)$. By
Claim \ref{claim:ap}, it follows that \be\label{eq:pairs}
b_0-b_1=a_0-a_1
%\mbox{ and } \|b_t-b_{t-1}\|=\|a_r-a_{r-1}\|
\mbox{ and }b_0-b_1'=a_0-a_1'. \ee Thus  $a_1=b_1=(1,0)$ and
$a'_1=b'_1=(0,1)$. By Claim \ref{claim:lat}, it follows in view of
$0\in A\cap B$ that $A$, $B$, and $A+B$ are contained in the integer
lattice. Moreover, in view of Claim \ref{claim:ap} and Claim
\ref{claim:lat} applied to $a_r$, it follows that
\be\label{eq:pairsbis} b_t-b_{t-1}=a_r-a_{r-1}=a_1-a_0=(1,0) \mbox{
and } a''_1\in A_1. \ee Figure 1 shows a picture of the situation.
\begin{figure}[ht]\label{caseone}
\setlength{\unitlength}{8mm}
\begin{center}
\begin{picture}(10,4)
\put(0,0){\line(1,0){10}} \put(0,1){\line(1,0){10}}
\put(0,0){\line(0,1){4}} %\put(8,0){\line(0,1){6}}
% \put(9,0){\line(0,1){6}}

 \put(0,0){\circle*{0.2}} \put(0,0){\circle{0.3}}
 \put(0,1){\circle*{0.2}} \put(0,1){\circle{0.3}}
 \put(1,0){\circle*{0.2}} \put(1,0){\circle{0.3}}
 \put(8,0){\circle*{0.2}}
 \put(9,0){\circle*{0.2}}
\put(-1,-0.7){\small $a_0=b_0$}\put(-1.7,0.9){\small $a_1'=b_1'$}
\put(0.8,-0.7){\small $a_1=b_1$} \put(7.5,-0.7){\small
$a_{r-1}$}\put(9,-0.7){\small $a_{r}=a''_0$} \put(7.9,1.2){\small
$a''_1$}\put(7.7,1){\circle*{0.2}} \put(-1.1,2.5){$e',
f'$}\put(4,-0.5){$e,
f$}\put(9,0){\line(-5,4){3}}\put(6.4,2.5){$f''$}
\end{picture}
\end{center}
\vspace{3mm} \caption{A picture of Case A.}
\end{figure}
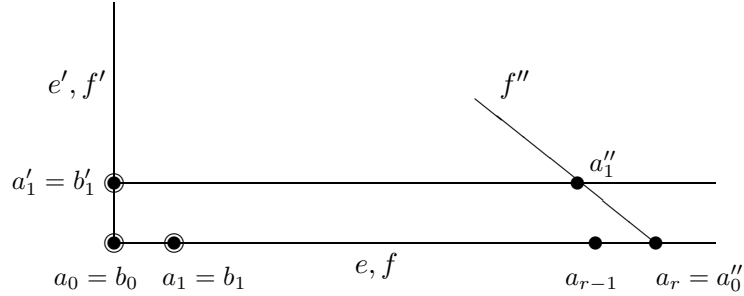
In view of Claim \ref{claim:par} and (\ref{eq:a0}), it follows that
$A\cup B$ is contained in the region defined by the lines $y=0$,
$x=0$ and the line defined by $f''$.

\textbf{Subcase A.1:}  $A_0$ is not in arithmetic progression. Thus
it follows, in view of the equality conditions for
%(\ref{eq:pairsbis}) that $|A_0|\geq 4$,
%implying $s\geq 5$, and it follows in view of
(\ref{eq:a0}), that \be\label{eq:super2}\phi_1(b_t)<\phi_1 (a_r).\ee
In view of Theorem \ref{CDT-for-Z} and the assumption of the
subcase, it follows that $|A_0+B_0|\geq |A_0|+|B_0|$. Hence \be
\label{eq:b<a}|B_0|<|A_0|\leq s-1,\ee since otherwise $|(A\setminus
A_0)+B|\le |A+B|-(|A_0|+|B_0|)\le |A+B|-2|A_0|$ yielding (b).
Consequently, \be\label{eq:super}\phi_1 (a''_1)\le \phi_1(a_r),\ee
since otherwise deletion of $A_0$ from $A$ and $B_0$ from $B$
decreases $A+B$ by at least
$$|A_0+B_0|+|A_0+b'_1|+|B_0+a''_1|\ge 2(|A_0|+|B_0|)$$ elements,
yielding (b) (note Claim \ref{claim:par} gives $|B_0|\leq |A_0|$).

If $ \phi_1(a_r)\leq 2s-4,$ then in view of (\ref{eq:super}) it
follows that $A\cup B$ is contained in the $2s-3$ vertical lines
$x=i$, $0\le i\le 2s-4$, and (a) holds. Therefore we may assume
$\phi_1(a_r)\ge 2s-3\geq |A_0|+|B_0|-1.$
%Hence, since $A$ and $B$ do not have $s$ collinear points, it
%follows in view of (\ref{eq:b<a}) that the number $h_a:=a_r-a_0+1$
%of holes in $A_0$ verifies $h_a\ge s-1> |B_0|$.
Since $\gcd(\phi_1(A_0))=1$, we can apply the Theorem
\ref{thm-Freiman-3k-3-stanchescu_version} to $A_0$ and $B_0$, with
$\delta=0$ in view of (\ref{eq:super2}). Thus, since
$\phi_1(a_r)\geq |A_0|+|B_0|-1,$ it follows that by removing the
elements of $A_0$ and $B_0$ from $A$ and $B$, respectively, we
decrease the cardinality of $A+B$ by at least
\be\label{eq:2s-2}|A_0+B_0|+|(A_0+B_1)\cup (B_0+A_1)|\ge
(|A_0|+2|B_0|-2)+|(A_0+B_1)\cup (B_0+A_1)|.\ee If $|B_1|\ge 2$,
then, from Theorem \ref{CDT-for-Z} and the assumption of the
subcase, it follows that $|A_0+B_1|\ge |A_0|+|B_1|\ge |A_0|+2$,
whence (\ref{eq:2s-2}) yields (b). Therefore $|B_1|=1$ and
$|(B_0+A_1)\setminus (A_0+B_1)|\le 1$. Consequently, \be
\label{trylinium}\phi_1 (a''_1)\le \phi_1(a_r)-\phi_1(b_{t-1}),\ee
with equality possible only if $a''_1+b_t$ is a unique expression
element in $A+B$.

Let $b$ be the intersection of $e''$ with the line $y=1$. By Claim
\ref{claim:par} and (\ref{eq:super}), the slope of $e''$ is no
steeper than the slope of $f''$. Hence (\ref{trylinium}) and
(\ref{eq:pairsbis}) yield \be\label{snakeye}\phi_1 (b_t)-\phi_1
(b)\ge \phi_1 (a_r)-\phi_1 (a''_1)\ge \phi_1 (b_{t-1})=\phi_1
(b_t)-1.\ee Consequently, $\phi_1(b)\leq 1$. If $\phi_1(b)=0$, then
it follows in view of (\ref{eq:b<a}) that $|B|=|B_0|+1\leq s-1$, a
contradiction. Therefore $\phi_1(b)>0$, which is only possible if
equality holds in (\ref{trylinium}), else the estimate from
(\ref{snakeye}) improves by $1$. Thus $a''_1+b_t$ is a unique
expression element, so that if $e''$ and $f''$ were parallel, then
by removing $a_r$ from $A$ and $b_t$ from $B$ we would lose the
elements $a_r+b_t$, $a_r+b_{t-1}=a_{r-1}+b_{t}$, $a''_1+b_t$ and
$a_r+b''_1$, yielding (b). So we may assume $e''$ and $f''$ are not
parallel, whence the estimate in (\ref{snakeye}) becomes strict,
yielding $0<\phi_1(b)<1$.

As a result, if $|B_0|\geq 3$, then (\ref{eq:b<a}) implies $|B|\leq
|B_0|+1\leq s-1$, a contradiction. Therefore $|B_0|=2$. Thus, since
$|A_0+B_0|\geq |A_0|+|B_0|=|A_0|+2$ and since $|(A_0\setminus
a_r)+(B_0\setminus b_t)|=|A_0\setminus a_r|$ (in view of
$|B_0\setminus b_t|=1$), it follows that removing $a_r$ from $A_0$
and $b_t$ from $B_0$ deletes at least three points from $A+B$
contained in $A_0+B_0$ as well as the unique expression element
$a''_1+b_t$, yielding (b), and completing the subcase.

\textbf{Subcase A.2:} $A_0$ is in arithmetic progression. We proceed
to verify that
 \be \label{eq:claim12} \phi_1 (a''_1)\le
\phi_1(a_r)+1 . \ee Suppose (\ref{eq:claim12}) is false. Since (b)
does not hold, it follows that \be\label{eq:a0b0}
|A_0+B_0|+|(A_0+B_1)\cup (A_1+B_0)|<2(|A_0|+|B_0|),\ee where the
left hand side is a lower bound for the number of elements deleted
from $A+B$ when removing $A_0$ from $A$ and $B_0$ from $B$. Since
$|(A_0+B_1)\cup (A_1+B_0)|\ge |A_0+b'_1|+|a''_1+B_0|$ (in view of
(\ref{eq:claim12}) not holding), we see that (\ref{eq:a0b0}) implies
$|A_0+B_0|=|A_0|+|B_0|-1$. Hence Theorem \ref{CDT-for-Z} implies
that both $A_0$ and $B_0$ are arithmetic progressions with the same
difference. Moreover, $|(A_0+B_1)\cup (A_1+B_0)|=
|A_0+b'_1|+|a''_1+B_0|$, whence (\ref{eq:claim12}) not holding
implies that $a_r+(1,1)\notin (A_0+B_1)\cup (A_1+B_0)$. From the
previous two sentences, we see that if $a_r=a+b_i$, with $a\in A_1$
and $i<t$, then $b_i+(1,0)=b_{i+1}\in B_0$ and
$a_r+(1,1)=a+b_i+(1,0)\in A_1+B_0$, a contradiction. Likewise, if
$a_r=a_i+b$, with $b\in B_1$, then $i=r$. As a result, we conclude
that $a_r+b'_1=a_r+(0,1)$ has at most two expressions in $A+B$, the
second one being possibly $a+b_t$ for some $a\in A_1$. Hence, by
deleting $a_r$ from $A_0$ and $b_t$ from $B_0$, we lose the four
elements $a_r+b_t,\, a_r+b_{t-1}=a_{r-1}+b_t,\, a_r+b'_1=a_r+(0,1)$,
and $z$, where $z$ is the element of $A+B$ contained on the line
$y=1$ with $\phi_1(z)$ maximal (note $\phi_1(z)\geq
\phi_1(a''_1+b_t)>\phi_1(a_{r}+b'_1)$). Thus  (b) follows, and so we
may assume that (\ref{eq:claim12}) does indeed hold.

We can now conclude Case A.  If either $A_0$ or $A'_0=A\cap \{x=0\}$
are not in arithmetic progression, then (a) holds by Subcase A.1
applied to either the lines $y=0$ or $x=0$. Otherwise, both $A_0$
and $A'_0$ are arithmetic progressions and, by (\ref{eq:claim12})
and (\ref{eq:a0}) applied both to the lines $x=0$ and $y=0$, it
follows in view of Claim \ref{claim:par} that $A\cup B$ is contained
in the at most $2s-3$ lines with slope $1$ passing through the
points in $A_0\cup A'_0$, yielding (a).

\textbf{Case B:} Claim \ref{claim:edge}(ii) holds for the pair $f$
and $f'$. This case is slightly simpler than Case A, and we use very
similar arguments. Recall that $b_0=(0,0)$, $b_1=(1,0)$,
$b'_1=(0,1)$ and both $A$ and $B$ are contained in the positive
first quadrant. We may also assume $f$ is contained in the
horizontal axis and $f'$ is contained in the vertical axis;
furthermore, by the same arguments used to establish
(\ref{eq:pairsbis}), we have $a_0=(1/d,0)$, $a_1=(1/d+1,0)$,
$a'_0=(0,1/d')$ and $a'_1=(0,1/d'+1)$, for some $d,\,d'\in \R^+$,
and $a''_1\in A_{1/d'}$. From Claim \ref{claim:lat} (applied both to
$f$ and $g$ and to $g$ and $f'$) we conclude $d,\,d'\in\Z^+$ and
that the lines defined by $a'_0$ and $a_0$ and by $a'_1$ and $a_1$
must be parallel, which implies $d=d'$ (Figure 2 illustrates the
argument);
\begin{figure}[ht]\label{casetwo}
\setlength{\unitlength}{3mm}
\begin{center}
\begin{picture}(12,8)
\put(0,0){\line(0,1){8}} \put(0,2){\line(3,-2){3}}
\put(3,1){\line(1,0){1}} \put(4,0){\line(0,1){1}}
\put(0,8){\line(3,-2){12}} \put(0,0){\line(1,0){12}}
\put(3,0){\line(0,1){6}}

\put(-2.5,.3){\small $1/d'$}  \put(.37,-1.5){\small $1/d$}

 \put(0,2){\circle*{0.3}} \put(-1.5,2.2){$a'_0$}
 \put(0,8){\circle*{0.3}} \put(-1.5,8){$a'_1$}
 \put(3,0){\circle*{0.3}} \put(3.2,-1.5){$a_0$}
 \put(12,0){\circle*{0.3}} \put(12,-1.5){$a_1$}
\put(-1,4.5){$1$} \put(2,3){$1$} \put(7.5,-1.5){$1$}
%\put(0,0){\circle{0.5}} \put(0,6){\circle{0.5}}
%\put(9,0){\circle{0.5}}

\end{picture}
\end{center}
\vspace{3mm} \caption{Why $d=d'$.}
\end{figure}
moreover, we have that $A+B$ is contained within the lattice
$(1/d,0)+\Z(1,0)+\Z(-1/d,1/d)$. As in Case A, we have $A$ contained
in the region defined by the lines $x=0$, $y=0$ and the line defined
by $f''$.

Since $A+B$ is contained within the lattice
$(1/d,0)+\Z(1,0)+\Z(-1/d,1/d)$, by removing $b_0$ from $B$ and $a_0$
and $a'_0$ from $A$, we lose all the elements of $A+B$ contained
within the two lines with slope $-1$ passing through $a_0$ and
$a_1$, i.e., all the elements from \ber\nn(b_0+\{a_0,a'_0\})\cup
(b_0+\{a_1,a'_1\})\cup
(\{a_0,a'_0\}+\{b_1,b'_1\})=\\\nn\{(0,1/d),(1/d,0), (1+1/d,0),
(0,1+1/d), (1,1/d), (1/d,1)\}.\eer If $d>1$, then the above $6$
elements are distinct, and (b) follows. Therefore we may assume
$d=1$. As a result, $b_0=(0,0)$, $a_0=b_1=(1,0)$, $a_1=(2,0)$
$a'_0=b_1=(1,0)$, $a'_1=(2,0)$, and $A$, $B$ and $A+B$ are contained
in the integer lattice.

Let us show that \be\label{eq:case2} \phi_1 (a''_1)\le \phi_1
(a_r).\ee Suppose on the contrary that (\ref{eq:case2}) does not
hold. Then it follows, in view of Theorem \ref{CDT-for-Z} and
$a''_1\in A_1$, that by removing $A_0$ from $A$ and $B_0$ from $B$
we lose at least \be\label{eq:case2a} |A_0+B_0|+|(A_0+B_1)\cup
(A_1+B_0)|\ge |A_0|+|B_0|-1+|b'_1+A_0|+|a''_1+B_0|+|\{a'_1+b_0\}|=
2(|A_0|+|B_0|) \ee elements from $A+B$, yielding (b). So we may
assume (\ref{eq:case2}) holds.

Now, if $\phi_1 (a_r)\le 2s-4$, then it follows, in view of
(\ref{eq:case2}), Claim \ref{claim:par} and (\ref{eq:a0}), that
$A\cup B$ is contained in the $2s-3$ parallel lines $x=i$, $0\le
i\le 2s-4$, yielding (a). Therefore we may assume $\phi_1 (a_r)\ge
2s-3$. Hence, since $2s-3\ge s$ for $s\ge 3$, it follows that $A_0$
is not in arithmetic progression. Furthermore, with the same
argument used to deduce (\ref{eq:b<a}), we conclude $|B_0|<|A_0|\leq
s-1$. The remainder of the proof is now just a simplification of
that of Case A.1, which proceeds as follows.

Since $\phi_1 (a_r)\ge 2s-3$ and $|A_0|>|B_0|$, we have
$\phi_1(a_r)-1\geq  |A_0|+s-3\geq |A_0|+|B_0|-1$. Thus, by the same
argument used in Case A, we conclude that (\ref{eq:2s-2}) holds. If
$|(A_0+B_1)\cup (A_1+B_0)|\geq |A_0|+2$, then (\ref{eq:2s-2})
implies (b). Therefore we may assume $|(A_0+B_1)\cup (A_1+B_0)|\leq
|A_0|+1$, and consequently, since $\{a'_1+b_0\}\cup
(b'_1+A_0)\subseteq (A_0+B_1)\cup (A_1+B_0)$ with $|\{a'_1+b_0\}\cup
(b'_1+A_0)|=|A_0|+1$, we conclude that $$(A_0+B_1)\cup
(A_1+B_0)=\{a'_1+b_0\}\cup (b'_1+A_0).$$ As a result,
$\phi_1(a''_1)+\phi_1(b_t)\leq \phi_1(a_r)$.

Let $b$ be the intersection of the edge $e''$ with the line $y=1$.
In view of in view of (\ref{eq:case2}) and Claim \ref{claim:par},
the slope of $e''$ is no steeper than that of $f''$. Thus, since
$\phi_1(a''_1)+\phi_1(b_t)\leq \phi_1(a_r)$, it follows that $\phi_1
(b_t)-\phi_1(b)\ge \phi_1 (a_r)-\phi_1(a''_1)\ge \phi_1 (b_t)$,
implying $\phi_1 (b)=0$. Hence $|B|\le |B_0\cup \{ b'_1\}|\le s-1$
(in view of $|B_0|<|A_0|$), a contradiction. This completes the
proof.
\end{proof}

The following lemma will allow us to improve, in a very particular
case, the bound given in Theorem \ref{THEbounds} by one, which will
be a crucial improvement needed in the proof of Theorem
\ref{final-result} for the extremal case $|A|+|B|\leq 4s^2-5s-1$.

\begin{lem}\label{the-monstruous-lemma} Let $X=(x_1,x_2)$ be a basis for
$\R^2$, let $s\geq 2$ be an integer, let $A,\,B\subseteq \R^2$ be
finite, nonempty subsets with $||A|-|B||\leq s$
 and $4s^2-6s+3\leq |A|+|B|\leq 4s^2-5s-1$.
Suppose that $|\phi_{X_1}(A)|\leq |\phi_{X_1}(B)|=2s-2$, where
$X_1=\R x_1$, and that some line parallel to $\R x_1$ intersects $A$
in at least $2s-2$ points. Then \be\label{onebitmore}|A+B|\geq
2|A|+2|B|-6s+7.\ee\end{lem}

\begin{proof} We may w.l.o.g. assume $\C_X(A)=A$ and $\C_X(B)=B$.
Let $m=|\phi_{X_1}(A)|$ and $n=|\phi_{X_1}(B)|$. Let $A_i=A\cap (\Z
x_1+(i-1)x_2)$, $B_j=(\Z x_1+(j-1)x_2)$, $|A_i|=a_i$ and
$|B_i|=b_i$, for $i=1,\ldots,m$ and $j=1,\ldots,n$. By hypothesis,
we have $a_1\geq 2s-2$ and $m\leq n= 2s-2$. Assume by contradiction
\be\label{onebitmore-cont}|A+B|\leq 2|A|+2|B|-6s+6.\ee

Suppose $m<n=2s-2$. Then, since $||A|-|B||\leq s\leq 2s-2$, from the
proof of Theorem \ref{THEbounds} we know that (\ref{la-lb-bound}) is
minimized for the boundary value $m=n-1$. Hence
$$|A+B|\geq
|A|+|B|-(n+n-1-1)+\frac{n-1}{n-1}|A|+\frac{n-2}{n}|B|=2|A|+2|B|-4s+6-\frac{2}{2s-2}|B|,$$
which together with (\ref{onebitmore-cont}) implies $|B|\geq
s(2s-2)$. Consequently, $|A|+|B|\geq 2|B|-s\geq 2s(2s-2)-s=4s^2-5s$,
contradicting our hypotheses. So we may assume $m=n=2s-2$.

Observe that,  for each $j=1,\ldots,s-1$, we have the following
estimates:
\begin{eqnarray} |A+B|+4s-5&\geq&
\Sum{i=1}{j-1}(a_i+b_1)+\Sum{i=1}{2s-2-j}(a_j+b_i)+\Sum{i=j}{2s-2}(a_i+b_{2s-j-1})+\Sum{i=2s-j}{2s-2}(a_{2s-2}
+b_i)\nonumber \\\label{ff1}
&=&|A|+|B|+(j-1)(a_{2s-2}+b_1)+(2s-2-j)(a_j+b_{2s-j-1}),
\end{eqnarray}
\begin{eqnarray}
|A+B|+4s-5&\geq&
\Sum{i=1}{j-1}(a_1+b_i)+\Sum{i=1}{2s-2-j}(a_i+b_j)+\Sum{i=j}{2s-2}(a_{2s-j-1}+b_{i})+\Sum{i=2s-j}{2s-2}(a_{i}
+b_{2s-2})\nonumber\\\label{ff2}
&=&|A|+|B|+(j-1)(b_{2s-2}+a_1)+(2s-2-j)(b_j+a_{2s-j-1}),
\end{eqnarray}
\begin{eqnarray}
\label{ee1} |A+B|+4s-5&\geq& \Sum{i=1}{2s-
3}(a_i+b_i+a_{i+1}+b_i)+a_{2s-2}+b_{2s-2}=2|A|+2|B|-a_1-b_{2s-2},\\\label{ee2}
|A+B|+4s-5&\geq& \Sum{i=1}{2s-
3}(a_i+b_i+a_{i}+b_{i+1})+a_{2s-2}+b_{2s-2}=2|A|+2|B|-b_1-a_{2s-2}.
\end{eqnarray}
In view of (\ref{onebitmore-cont}) and (\ref{ff1}) with $j=1$, it
follows that $|A|+|B|\geq (2s-3)(a_1+b_{2s-2})+2s-1$. Thus
$|A|+|B|\leq 4s^2-5s-1$ implies that $a_1+b_{2s-2}\leq 2s-1$.
However, in view of (\ref{ee1}) and (\ref{onebitmore-cont}), it
follows that $a_1+b_{2s-2}\geq 2s-1$. Consequently, \be\label{ff3}
a_1+b_{2s-2}=2s-1.\ee Repeating these arguments with (\ref{ff2}) and
(\ref{ee2}) instead, we likewise conclude \be\label{ff4}
b_1+a_{2s-2}= 2s-1.\ee If $a_j+b_{2s-j-1}\geq 2s$, then, in view of
(\ref{ff4}), (\ref{onebitmore-cont}) and (\ref{ff1}), it follows
that $$|A|+|B|\geq j(2s-1)+(2s-2-j)(2s)=4s^2-4s-j\geq 4s^2-5s+1,$$
contradicting that $|A|+|B|\leq 4s^2-5s-1$. Therefore we may assume
\be\label{pluppy} a_j+b_{2s-j-1}\leq 2s-1,\ee for all
$j=1,\ldots,s-1$. Repeating this argument with (\ref{ff2}) and
(\ref{ff3}) instead, we likewise conclude \be\label{pluppy2}
b_j+a_{2s-j-1}\leq 2s-1,\ee for all $j=1,\ldots,s-1$. However,
summing (\ref{pluppy}) and (\ref{pluppy2}) over $j=1,\ldots,s-1$
yields $$|A|+|B|\leq 2(s-1)(2s-1)=4s^2-6s+2,$$ contradicting our
hypotheses, and completing the proof.
\end{proof}

The proof of Theorem \ref{final-result} is by induction on $s$ and
it uses the following version, which is essentially equivalent to
Theorem \ref{final-result}.

\begin{thm}\label{david-part} Let $s\geq 3$ be an integer, and let
$A,\,B\subseteq \mathbb{R}^2$ be finite subsets such that there are
no $s$ collinear points in either $A$ or $B$.

{\rm (i)}  If $||A|-|B||\leq s$ and $|A|+|B|\geq (s-1)(4s-6)+1$,
then
$$|A+B|\geq 2|A|+2|B|-6s+7.$$

{\rm (ii)} If $|A|\geq |B|+s$ and $|B|\geq \frac{1}{2}(s-1)(4s-7)$,
then
$$|A+B|\geq |A|+3|B|-5s+7.$$
\end{thm}

We first show that part (ii), in both Theorem \ref{david-part} and
\ref{final-result}, is a very simple consequence of the
corresponding part (i).

\begin{lem}\label{oopsie} Let $s\geq 2$ be a positive integer. (a) If $s\geq 3$ and Theorem
\ref{david-part}(i) holds for $s$, then Theorem \ref{david-part}(ii)
holds for $s$. (b) If Theorem \ref{final-result}(i) holds for $s$,
then Theorem \ref{final-result}(ii) holds for $s$.
\end{lem}

\begin{proof} We first prove (a).
Observe that $|(A\setminus x)+B|<|A+B|$ for any vertex $x$ in the
convex hull of $A$. Thus, by iteratively deleting vertices from the
convex hull, we can obtain a subset $A'\subseteq A$ with
$|A'|=|B|+s$ and \be\label{points-deleted}|A'+B|\leq
|A+B|-|A\setminus A'|.\ee %Furthermore, since at each step we can
%delete any vertex from the convex hull, it follows (in view of
%$|A'|\geq |B|+s\geq 4$) that we can choose the vertices so that the
%set $A'$ remains $2$-dimensional.

Since $|B|\geq \frac{1}{2}(s-1)(4s-7)$, it follows that
$|A'|+|B|=2|B|+s\geq (s-1)(4s-6)+1$, whence we can apply Theorem
\ref{david-part}(i) to $A'+B$. Thus $|A'+B|\geq
2|A'|+2|B|-6s+7=|A'|+3|B|-5s+7$, whence the theorem follows in view
of (\ref{points-deleted}).

Next we prove (b). Suppose by contradiction that $h_1(A,B)\geq s$.
%Observe that for $|A|\geq |B|+s$ we have
%$2|A|+2|B|-2s+1-\frac{|A|+|B|}{s}\geq |A|+3|B|-s-\frac{2|B|}{s}.$
%Thus in view of Theorem \ref{THEbounds}(a) it follows that
%$h_1(A,B)\geq \lfloor\frac{|A|+|B|}{2s}\rfloor+1\geq
%\lfloor\frac{2|B|+s}{2s}\rfloor+1\geq
%\lfloor\frac{4s^2-6s+3}{2s}\rfloor+1=2s-2$. However, the more
%detailed calculation from the end of Case A (equation
%(\ref{sneekay}) and Lemma \ref{the-monstruous-lemma}) shows that
%$h_1(A,B)\geq 2s-1$.
As in the previous part, observe that $|(A\setminus x)+B|<|A+B|$ for
any vertex $x$ in the convex hull of $A$. Thus by iteratively
deleting vertices from the convex hull we can obtain a sequence of
subsets $A_0=A\supseteq A_1\supseteq \ldots \supseteq
A_{|A|-|B|-s}=A_k$, with $|A_i|=|A|-i$
%with each
%$A_i$ $2$-dimensional (since $|A_i|\geq |B|+s\geq 4$),
and \be\label{noy}|A_i+B|\leq |A+B|-|A\setminus
A_i|<|A_i|+3|B|-s-\frac{2|B|}{s},\ee where the last inequality
follows from (\ref{speel}).

Since $|A_i|=|A_{i-1}|-1$ and $A_i\subseteq A_{i-1}$, it follows
that $h_1(A_i,B)\geq h_1(A_{i-1},B)-1$ for all $i$. Consequently, if
$h(A_k,B)< s$, then it would follow in view of $h(A,B)\geq s$ that
$h(A_j,B)=s$ for some $j$, whence Theorem \ref{THEbounds}(i)(ii)
would contradict (\ref{noy}) for $i=j$ (note the bound in Theorem
\ref{THEbounds}(i) implies that in Theorem \ref{THEbounds}(ii) in
view of $|A_j|\geq |A_k|= |B|+s$). Therefore we may assume
$h(A_k,B)\geq s$.

Since $|B|\geq 2s^2-\frac{7}{2}s+\frac{3}{2}$, it follow that
$|A_k|+|B|=2|B|+s\geq 4s^2-6s+3$. Hence we can apply Theorem
\ref{final-result}(i) to $A_k+B$, whence $h_1(A_k,B)\geq s$ implies
$$|A_k+B|\geq
2|A_k|+2|B|-2s+1-\frac{|A_k|+|B|}{s}=|A_k|+3|B|-s-\frac{2|B|}{s},$$
contradicting (\ref{noy}) for $i=k$, and completing the proof.
\end{proof}

We will prove Theorems \ref{david-part} and \ref{final-result}
simultaneously using an inductive argument on $s$: the case $s-1$ of
Theorem \ref{final-result} will be used to prove the case $s$ of
Theorem \ref{david-part}, while the case $s$ of Theorem
\ref{david-part} will be used to prove the case $s$ of Theorem
\ref{final-result} (except for the case $s=2$, where a trivial
argument will be used instead). Thus both Theorem \ref{david-part}
and \ref{final-result} follow immediately from the following two
lemmas. This also shows that Theorem \ref{david-part} and Theorem
\ref{final-result} are in some sense equivalent statements.

\begin{lem} Let $s\ge 3$ be a positive integer. Suppose that the statement in Theorem
\ref{final-result} holds for $s-1$. Then Theorem \ref{david-part}
holds for $s$.
\end{lem}

\begin{proof} In view of Lemma \ref{oopsie}, it suffices to show part
(i) holds, so suppose on the contrary that Theorem
\ref{david-part}(i) is false for $s$.  Let $A,\,B\subseteq \R^2$ be
a counterexample with $|A|+|B|$ minimum. Thus  $||A|-|B||\leq s$,
$|A|+|B|\geq (s-1)(4s-6)+1$ and
    \be\label{toosmall}
    |A+B|<2|A|+2|B|-6s+7.
    \ee
We may assume $|A|\geq |B|$.

Since neither $A$ nor $B$ contains $s$ collinear points, and since
$|A|+|B|\geq (s-1)(4s-6)+1$, it follows from the pigeonhole
principle that $h_1(A,B)>2s-3$.  By Lemma \ref{lem-stan-reduction}
(in view of (\ref{toosmall})), there is a nonempty subset
$A_0\subseteq A$ and $B_0\subseteq B$ with $ |B_0|\le |A_0|\le s-1$
and
    \be\label{oneextra}
    |A'+B'|\leq |A+B|-2(|A_0|+|B_0|)<2|A'|+2|B'|-6s+7,
    \ee
where $A'=A\setminus A_0$ and $B'=B\setminus B_0$. Furthermore,
$||A'|-|B'||=||A|-|B|-(|A_0|-|B_0|)|\leq s$. Therefore, by the
minimality of $|A|+|B|$, we have
    $$
    |A'|+|B'|\leq (s-1)(4s-6).
    $$
As a result,
    \be\label{umum2}
    |A|+|B|\leq |A'|+(s-1)+|B'|+(s-1)\le (s-1)(4s-6)+2(s-1)=4(s-1)^2.
    \ee
If $|A|<|B|+s$, then, since $h_1(A,B)> 2s-3\geq s-1$ and since
    $$
    |A|+|B|\geq (s-1)(4s-6)+1>(s-1)(4s-9)+3=4(s-1)^2-5(s-1)+3,
    $$
it follows, in view of (\ref{toosmall}) and the case $s-1$ of
Theorem \ref{final-result}(i), that
   $$
    2|A|+2|B|-2(s-1)+1-\frac{|A|+|B|}{s-1}\leq |A+B|\leq 2|A|+2|B|-6s+6.
    $$
Hence $|A|+|B|\geq (4s-3)(s-1)>4(s-1)^2$, contradicting
(\ref{umum2}). On the other hand, if $|A|=|B|+s$, then, since
$h_1(A,B)> 2s-3\geq s-1$ and since
    \ber\nn
    2|B|+s&=&|A|+|B|\geq (s-1)(4s-6)+1=4s^2-10s+7\\ &\geq&\nn 4s^2-14s+14= 4(s-1)^2-7(s-1)+3+s,
    \eer
it follows, in view of (\ref{toosmall}) and the case $s-1$ of
Theorem \ref{final-result}(ii), that
    $$
    2|A|+2|B|-2s+1-\frac{|A|+|B|-s}{s-1}=|A|+3|B|-(s-1)-\frac{2|B|}{s-1}\leq |A+B|\leq 2|A|+2|B|-6s+6.
    $$
Hence $|A|+|B|\geq (4s-5)(s-1)+s=4s^2-8s+5>4(s-1)^2$, contradicting
(\ref{umum2}), and completing the proof.
\end{proof}

\begin{lem} Let $s\geq 2$ be a positive integer. If $s\geq 3$, suppose that the
statement of Theorem \ref{david-part} holds for $s$. Then Theorem
\ref{final-result} holds for $s$.
\end{lem}

\begin{proof}  In view of Lemma \ref{oopsie}, it suffices to show part
(i) holds. Let $A,\,B\subseteq \R^2$ verify the hypothesis of
Theorem \ref{final-result}(i) for $s$, and assume by contradiction
that $h_1(A,B)\geq s$.

Suppose neither $A$ nor $B$ contain $s$ collinear points. Thus
$|A|+|B|\geq 3$ implies that $s\geq 3$. Hence, in view of Theorem
\ref{david-part}(i) and (\ref{sim-assumption}), it follows that
    $$
    2|A|+2|B|-6s+7\leq |A+B|<2|A|+2|B|-2s+1-\frac{|A|+|B|}{s}.
    $$
Thus $|A|+|B|<4s^2-6s$, contradicting that $|A|+|B|\geq 4s^2-6s+3$.
So we may assume w.l.o.g. that $A$ contains at least $s$ collinear
points on the line $\Z x_1+a_1$. Let $X=(x_1,x_2)$ be an ordered
basis for $\R^2$
%By an affine transformation we may
%assume that $A$ has at least $s$ points on the line $\{x=0\}$.
%
%you need to more careful about assuming w.l.o.g. that the sets are compressed.
% h_2(A,B)\neq h_2(C_X(A),C_X(B)) in general.

Since $h_1(A,B)\geq s$, so that
$\max\{|\phi_{X_1}(A)|,\,|\phi_{X_1}(B)|\}\geq s$, it follows in
view of (\ref{compression-doesn't-affect-param}) that
$\max\{|\phi_{X_1}(\C_X(A))|,\,|\phi_{X_1}(\C_X(B))|\}\geq s.$
Hence, since $A$ contains $s$ collinear points on a line parallel to
$\Z x_1$, it follows that $h_1(\C_X(A),\C_X(B))\geq s$.
Consequently, we conclude from
(\ref{well-align-lowerbound-fullycompressed}) that it suffices to
prove the theorem on compressed sets, and w.l.o.g. we assume
$A=\C_X(A)$ and $B=\C_X(B)$. Let $|\phi_{X_1}(A)|=m$ and
$|\phi_{X_1}(B)|=n$. Let $A_i=A\cap (\Z x_1+(i-1)x_2)$, $1\le i\le
m$, and $B_i=B\cap (\Z x_1+(i-1)x_2)$, $1\le i\le n$. Note, since
both $A$ and $B$ are compressed, that $|A_1|\geq |A_2|\geq
\ldots\geq |A_{m}|$ and $|B_1|\geq |B_2|\geq \ldots \geq |B_{n}|$.
Since $A$ contains $s$ collinear points along a line parallel to $\Z
x_1$, it follows that $|A_1|\geq s$.

By our assumption to the contrary, we have $\max\{m,\,n\}\geq s$.
Thus it follows, from Theorem \ref{THEbounds}(i) (applied with the
line $\Z x_1$) and (\ref{sim-assumption}), that
    \be\label{sbig}
    \max\{m,\,n\}\geq \left\lfloor\frac{|A|+|B|}{2s}\right\rfloor+1.
    \ee
Since $\max\{|A_1|,\,|B_1|\}\geq s$, it follows, from Theorem
\ref{THEbounds}(i) (applied with the line $\Z x_2$) and
(\ref{sim-assumption}), that
    \be\label{sbigdual}
    \max\{|A_1|,\,|B_1|\}\geq \left\lfloor\frac{|A|+|B|}{2s}\right\rfloor+1.
    \ee

 Let $k=|A|+|B|$, and let
    $$
    x=\left\lfloor\frac{|A|+|B|}{2s}\right\rfloor+1=\frac{|A|+|B|-\alpha}{2s}+1,
    $$
so that $k=|A|+|B|\equiv \alpha \mod 2s$, with $0\leq \alpha\leq
2s-1$. With this notation, (\ref{sim-assumption}) yields
    \be\label{a+b-upper-2x}
    |A+B|\leq 2(k-s-x+1)-\delta,
    \ee
where $\delta=0$ if $\alpha<s$ and otherwise $\delta=1$.

We proceed to show that \be|A+B|<
k-(2x-2)+\frac{x-2}{x}|A|+|B|.\label{simpler}\ee Suppose
(\ref{simpler}) does not hold.  In this case, if $\delta=0$, then
$\alpha\leq s-1$ whence from (\ref{a+b-upper-2x}) we conclude that
    $$
    |A|\geq sx=s(\frac{|A|+|B|-\alpha}{2s}+1)\geq s(\frac{2|A|-s-\alpha}{2s}+1)\geq
    s(\frac{2|A|-2s+1}{2s}+1)>|A|,
    $$
a contradiction. On the other hand, if $\delta=1$, then from
(\ref{a+b-upper-2x}) we instead conclude that
    $$
    2|A|\geq(2s+1)x\geq (2s+1)\frac{|A|+|B|+1}{2s}\geq
    (2s+1)\frac{2|A|-s+1}{2s},$$
whence
    \be\label{nunce2}
    |A|\leq s^2-\frac{s}{2}-\frac{1}{2}.
    \ee
However, since $2|A|+s\geq |A|+|B|\geq 4s^2-6s+3$, it follows that
$|A|\geq \lceil 2s^2-\frac{7}{2}s+\frac{3}{2}\rceil$, which
contradicts (\ref{nunce2}). Thus we conclude that (\ref{simpler})
holds.

For each $r\in \{1,\ldots,n\}$, we have the estimate
    \begin{eqnarray}\label{donce2}
    |A+B|&\geq&|A_1+\bigcup_{i=1}^{r-1}B_i|+|A+B_{r}|+|A_{m}+\bigcup_{i=r+1}^{n}B_i|\nonumber\\
        &=&\Sum{i=1}{r-1}|B_i|+(r-1)(|A_1|-1)+|A|+m(|B_r|-1)+\Sum{i=r+1}{n}|B_i|+(n-r)(|A_{m}|-1)\nonumber\\
        &\ge& |A|+|B|-1+(|A_1|-1)(r-1)+(m-1)(|B_r|-1).
    \end{eqnarray}
Averaging this estimate over all $r$, we obtain
    \be\label{(l+1)/2-bound}
    |A+B|\geq |A|+|B|-1+(|A_1|-1)(\frac{n+1}{2}-1)+(m-1)(\frac{|B|}{n}-1).
    \ee

In view of (\ref{sbig}) and (\ref{sbigdual}), we have
$\max\{m,\,n\}\ge x$ and  $\max\{|A_1|,\,|B_1|\}\ge x$. We consider
two cases according to whether these maxima are achieved in the same
set or in different sets.

 \textbf{Case A:} Either $\min\{ m,\,|B_1|\}\geq x$ or
$\min\{ n,\,|A_1|\}\geq x$. By symmetry we may assume  that the
latter holds. We have the estimate
    \begin{eqnarray}\label{once2}
    |A+B|&\geq& |A_1+(B\setminus B_{n})|+|A+B_{n}|\nonumber\\
    &=&|B|-|B_{n}|+(n-1)(|A_1|-1)+|A|+m(|B_{n}|-1)\nonumber\\
    &\ge&|A|+|B|-1+(n-1)(|A_1|-1)\nonumber\\
    &\geq& |A|+|B|-1+(x-1)^2.
\end{eqnarray}
In view of (\ref{a+b-upper-2x}) and (\ref{once2}), it follows that
    $$
    k\geq x^2+2s-2+\delta=\frac{k^2-2\alpha k+\alpha^2}{4s^2}+\frac{k-\alpha}{s}+2s-1+\delta.
    $$
Hence,
    $$
    k^2-2(2s^2-2s+\alpha)k+(8s^3-4s^2+4\delta s^2-4\alpha s+\alpha^2)\leq 0.
    $$
Thus, since $\alpha-\delta\leq 2s-2$, it follows that
    $$
    k\leq 2s^2-2s+\alpha+2s\sqrt{s^2-4s+2+\alpha-\delta}< 4s^2-4s+\alpha.
    $$
Since $|A|+|B|\equiv \alpha\mod 2s$, the above bound implies that
\be\label{verbilityX}|A|+|B|=k\leq 4s^2-6s+\alpha\leq 4s^2-4s-1.\ee
Hence, since $k\geq 4s^2-6s+3$, it follows that $k=4s^2-6s+\alpha$,
with $\alpha\geq 3$ and $x=2s-2$.

Suppose  $\max\{m,n\}=x$. If $\alpha<s$, then Lemma
\ref{the-monstruous-lemma} contradicts (\ref{a+b-upper-2x}).
Therefore $\alpha\geq s$ and $\delta =1$. Hence Theorem
\ref{THEbounds}(i) and (\ref{a+b-upper-2x}) imply that
\be\label{sneekay}2k-2x-2s+1\geq
2k-2x+1-\left\lfloor\frac{k}{x}\right\rfloor=2k-2x+1-(2s-1),\ee a
contradiction. So we may assume $\max\{m,n\}>x$.

Suppose $n\geq x+1$. Hence (\ref{once2}) now implies that $|A+B|\geq
|A|+|B|-1+x(x-1)$, which, when combined with (\ref{a+b-upper-2x})
and $x=2s-2$, yields $k\geq 4s^2-4s-1+\delta$, contradicting
(\ref{verbilityX}). So we can assume $n=x$ and $m>x$. By this same
argument, we also conclude that $|A_1|=x$.

If $|B_1|\geq x$, then interchanging the roles of $A$ and $B$ and
repeating the above argument completes the proof. Therefore
$|B_1|\leq x-1$. Since $|A_1|=x$, we can apply (\ref{la-lb-bound})
with the line $\Z x_2$ to obtain
$$|A+B|\geq \left( \frac{|A|}{x}+\frac{|B|}{|B_1|}-1\right)(x+|B_1|-1)=
k-(x+|B_1|-1)+\frac{|B_1|-1}{x}|A|+\frac{x-1}{|B_1|}|B|.$$
Considering this bound as a function of $|B_1|$, it follows by the
same calculation used in the proof of Theorem \ref{THEbounds}, and
in view of $|B_1|<x$ and $||A|-|B||\leq s\leq 2s-2=x$, that it is
minimized when $|B_1|=x-1$, contradicting (\ref{simpler}), and
completing the case.

\textbf{Case B:} Either $\min\{ m,\,|A_1|\}\geq x$ or $\min\{
n,\,|B_1|\}\geq x$. By symmetry we may assume that the former holds.
Note that we can assume $|B_1|<x$ and $n<x$, else the previous case
completes the proof.

If $m=x$, then, in view of $n\leq x-1$ and $||A|-|B||\leq s\leq
x=m$, it follows that the bound given by (\ref{la-lb-bound}),
considered as a function of $n$, is minimized for the boundary value
$n=x-1$, contradicting (\ref{simpler}). Therefore we may assume
$m>x$. Applying the same arguments with the roles of $x_1$ and $x_2$
swapped, we also conclude that $|A_1|>x$. Thus (\ref{(l+1)/2-bound})
implies that
    $$
    |A+B|\geq |A|+|B|-1+\frac{1}{2}x(n+1+\frac{2|B|}{n})-2x\geq
    k-1+x(\sqrt{2|B|}+\frac{1}{2})-2x.
    $$
Hence in view of (\ref{a+b-upper-2x}), it follows that
    \be\label{fashizzle}
    x(\sqrt{2|B|}+\frac{1}{2})\leq k-\delta-2s+3,
    \ee
and consequently,
$$(\frac{k-2s+1}{2s}+1)(\sqrt{2|B|}+\frac{1}{2})\leq k-2s+3.$$
Thus $\sqrt{2|B|}+\frac{1}{2}< 2s$, implying that $|B|\leq 2s^2-s$,
whence $|A|+|B|\leq 4s^2-s$. As a result,
\be\label{datstuff}x=\left\{
                                                \begin{array}{ll}
                                                  2s, & 4s^2-2s\leq k\leq 4s^2-s \\
                                                  2s-1, & 4s^2-4s\leq k\leq 4s^2-2s-1 \\
                                                  2s-2, & 4s^2-6s+3\leq k\leq 4s^2-4s-1.
                                                \end{array}
                                              \right.\ee
There are three cases based on the value of $x$.

If $x=2s$, then (\ref{datstuff}) implies that $k-\delta\leq
4s^2-s-1$, whence (\ref{fashizzle}) implies $$k\leq 2|B|+s\leq
(2s-2+\frac{1}{s})^2+s\leq 4s^2-7s+8,$$ contradicting that $k\geq
4s^2-2s$.

If $x=2s-1$, then (\ref{datstuff}) implies that $k-\delta\leq
4s^2-2s-2$, whence (\ref{fashizzle}) implies that $$k\leq 2|B|+s\leq
\lfloor(2s-\frac{3}{2})^2+s\rfloor\leq 4s^2-5s+2.$$ Hence $k\geq
4s^2-4s$ implies that $s=2$, whence the above inequality becomes
$k\leq 4s^2-5s+2=8$. Thus (\ref{fashizzle}) then implies that $k\leq
2|B|+s\leq (\frac{7}{3}-\frac{1}{2})^2+2\leq 6$, contradicting that
$k\geq 4s^2-4s=8$.

Finally, if $x=2s-2$, then (\ref{datstuff}) implies that
$k-\delta\leq 4s^2-4s-2$, whence (\ref{fashizzle}) implies
$$k\leq 2|B|+s\leq \lfloor(2s-\frac{3}{2}-\frac{1}{2s-2})^2+s\rfloor\leq 4s^2-5s.$$
%Hence, since $2|B|$ is an even integer, it follows that $k\leq
%2|B|+s\leq 4s^2-5s$.
However, $k\leq 4s^2-5s$ and (\ref{fashizzle})
imply that $k\leq 2|B|+s\leq (2s-2)^2+s=4s^2-7s+4$, contradicting
that $k\geq 4s^2-6s+3$, and completing the proof.
\end{proof}

Finally, we conclude with the proof of Theorem \ref{Thm-AisBig}.

\begin{proof}{\it of Theorem \ref{Thm-AisBig}.}
If $s=1$, then the result follows from Theorem \ref{CDT-for-Z}. If
$s=2$, then $|A|>|B|$, and the result follows from \cite[Corollary
5.16 with $n=|A|$, $t=|A|-|B|\geq 1$, $d=2$]{taobook}. So we may
assume $s\geq 3$. If $|B|=1$, the result is trivial. So $|B|\geq 2$.
By hypothesis,  \be\label{ideal-A-bound} |A|\geq
\frac{1}{2}s(s-1)|B|+s.\ee

Let $X=(x_1,x_2)$ be an arbitrary ordered basis for $\R^2$, where
$\R x_1=Z_1$ and $\R x_2=Z_2$. Let $m=|\phi_{Z_1}(A)|$ and
$n=|\phi_{Z_1}(B)|$. Note $\max\{m,\,n\}\geq s$ by hypothesis.

Suppose $m<s$. Then $n\geq s>m$ with $|B|<|A|$, whence Theorem
\ref{THEbounds}(i) implies that \be\label{noia}|A+B|\geq
2|A|+2|B|-2n+1-\frac{|A|+|B|}{n}.\ee Note (\ref{ideal-A-bound}) and
$s\geq 3$ imply $|A|\geq 3|B|+s$ so that $2\leq s\leq n\leq |B|\leq
\frac{|A|+|B|}{4}$. As a result, (\ref{noia}) and
(\ref{ideal-A-bound}) yield \ber\nn|A+B|&\geq&
2|A|+2|B|-3-\frac{|A|+|B|}{2}\geq
|A|+\frac{\frac{1}{2}s(s-1)|B|+s}{2}+\frac{3}{2}|B|-3 \\&=&
|A|+(\frac{1}{4}s^2-\frac{1}{4}s+\frac{3}{2})|B|+\frac{s}{2}-3\geq
|A|+(\frac{1}{4}s^2-\frac{1}{4}s+\frac{3}{2})|B|-s\geq
|A|+s|B|-s,\nn \eer as desired. So we may assume
$|\phi_{Z_1}(A)|=m\geq s$. Moreover, if $m=s$, then
(\ref{Goool-Aisbig}) follows in view of Theorem \ref{THEbounds}(iii)
and (\ref{ideal-A-bound}). Therefore $|\phi_{Z_1}(A)|=m>s$.  Since
$X$ was arbitrary, this means that $|\phi_Z(A)|>s$ for any
one-dimensional subspace $Z$. In particular, by letting $Z$ be a
line such that $|\phi_Z(B)|<|B|$ (recall $|B|\geq 2$), we conclude
 that $|A+B|\geq |A|+|\phi_Z(A)|\geq
|A|+s$. Thus we may assume $|B|\geq 3$, else the proof is complete.

If $n=1$, then (\ref{Goool-Aisbig}) follows from (\ref{la-lb-bound})
and $m>s$. Therefore, as $X$ is arbitrary, it follows that $n\geq 2$
and that $|\phi_Z(B)|\geq 2$ for any one-dimensional subspace $Z$.

Now assume to the contrary that (\ref{Goool-Aisbig}) is false. We
will throughout the course of the proof find that the following
bound holds for varying values of $n'\geq 1$:
\be\label{Bound-unmanipulatedform}
|A|+|B|-m-n'+1+\frac{n'-1}{m}|A|+\frac{m-1}{n'}|B|\leq |A+B|\leq
|A|+s|B|-s-1.\ee Inequality (\ref{la-lb-bound}) shows that the lower
bound above holds with $n'=n$. Rearranging the terms in
(\ref{Bound-unmanipulatedform}), we obtain
\be\label{Bound-m-quadform-original}
(\frac{|B|}{n'}-1)m^2-(s|B|-|B|+\frac{|B|}{n'}+n'-s-2)m+(n'-1)|A|\leq
0.\ee Applying the estimate (\ref{ideal-A-bound}) yields
\be\label{bound-oops}(\frac{|B|}{n'}-1)m^2-(s|B|-|B|+\frac{|B|}{n'}+n'-s-2)m+
(n'-1)(\frac{1}{2}s(s-1)|B|+s)\leq 0.\ee  When $|B|>n'$, the
discriminant of the above quadratic in $m$ must be nonnegative,
i.e.,
\be\label{the}(s|B|-|B|+M-s-2)^2-2(|B|+1-M)(s^2|B|-s|B|+2s)\geq
0,\ee where $M:=\frac{|B|}{n'}+n'$. Collecting terms, we obtain
\be\label{duhdumm}M^2+(2s^2|B|+2s-2|B|-4)M+4+4|B|-4s^2|B|+|B|^2-s^2|B|^2-4s|B|+s^2\geq
0.\ee Noting that $(2s^2|B|+2s-2|B|-4)\geq 0$, we conclude that
(\ref{duhdumm}) must hold for the maximum allowed value for $M$.

\setcounter{claim}{0}

\begin{claim}\label{n-not-2} (\ref{Bound-unmanipulatedform}) cannot hold with
$n'=2$; consequently, $|\phi_Z(B)|\geq 3$ for any one-dimensional
subspace $Z$.\end{claim}
\begin{proof} We know that (\ref{Bound-unmanipulatedform}) holds with
$n'=n$. Thus we need only prove the first part of the claim. Suppose
to the contrary that (\ref{Bound-unmanipulatedform}) holds with
$n'=2$. Thus considering (\ref{Bound-m-quadform-original}) as a
quadratic in $m$, we conclude that the discriminant is nonnegative,
i.e., that \ber|A|&\leq&\label{partduo}
\frac{(s|B|-\frac{|B|}{2}-s)^2}{2|B|-4}=\frac{(2s-1)^2|B|-4(2s-1)s|B|+4s^2}{8|B|-16}\\&=&
\frac{1}{8}(2s-1)^2|B|-\frac{1}{4}(2s-1)+\frac{(s-1)^2}{2|B|-4},\label{partuno}\eer
which contradicts the hypothesis of (a). Thus we may assume the
hypothesis of (b) holds. From (\ref{bound-oops}), we have
\be\label{parttree}(|B|-2)m^2-(2s|B|-|B|-2s)m+s(s-1)|B|+2s\leq 0.\ee
Considering (\ref{parttree}) as a quadratic in $m$, we see that its
minimum occurs for
$$m=\frac{(2s-1)|B|-2s}{2|B|-4}=s-\frac{1}{2}+\frac{s-1}{|B|-2}.$$ However, the
hypothesis $|B|\geq \frac{2s+4}{3}$ of (b) implies that
$s-\frac{1}{2}+\frac{s-1}{|B|-2}\leq s+1$. Consequently, since
$m\geq s+1$, we conclude that (\ref{parttree}) is minimized for the
boundary value $m=s+1$, whence $$0\geq
(|B|-2)(s+1)^2-(2s|B|-|B|-2s)(s+1)+s(s-1)|B|+2s=2|B|-2,$$
contradicting that $|B|\geq 3$, and completing the claim.\end{proof}

\begin{claim}\label{n'-not-3} If (\ref{Bound-unmanipulatedform}) holds with $n'=3$, then
$|B|\leq 6$; consequently, if $|B|\geq 7$, then $|\phi_Z(B)|\geq 4$
for any one-dimensional subspace $Z$.\end{claim}

\begin{proof}
As in the previous claim, we need only prove the first part.
Assuming (\ref{Bound-unmanipulatedform}) holds with $n'=3$, so that
$M=\frac{|B|}{3}+3$, it follows in view of (\ref{duhdumm}) and
$s\geq 3$ that \ber 0&\leq&
-s^2|B|^2-10s|B|+6s^2|B|+\frac{4}{3}|B|^2-4|B|+3+18s+3s^2\label{uuy}\\\nn&\leq&
-s^2|B|^2+6s^2|B|+\frac{4}{3}|B|^2+3s^s=-(\frac{23}{27}+\frac{4}{27})s^2|B|^2
+6s^2|B|+\frac{4}{3}|B|^2+3s^2\\\nn&\leq& -\frac{23}{27}s^2|B|^2
+6s^2|B|+3s^2,\eer which implies $|B|\leq 7$. However, it can be
individually checked that (\ref{uuy}) cannot hold for $|B|=7$,
completing the claim.
\end{proof}

\begin{claim}\label{n'-not-4} If (\ref{Bound-unmanipulatedform}) holds with $n'=4$, then
$|B|\leq 8$; consequently, if $|B|\geq 9$, then $|\phi_Z(B)|\geq 5$
for any one-dimensional subspace $Z$.\end{claim}

\begin{proof}
Assuming (\ref{Bound-unmanipulatedform}) holds with $n'=4$, so that
$M=\frac{|B|}{4}+4$, it follows in view of (\ref{duhdumm}) and
$s\geq 3$ that \ber\label{er}
0&\leq&-s^2|B|^2-7s|B|+8s^2|B|+\frac{9}{8}|B|^2-6|B|+8+16s+2s^2\\\nn&<&
-s^2|B|^2+8s^2|B|+\frac{9}{8}|B|^2+2s^2=-(\frac{7}{8}+\frac{1}{8})s^2|B|^2
+8s^2|B|+\frac{9}{8}|B|^2+2s^2\\\nn&\leq& -\frac{7}{8}s^2|B|^2
+8s^2|B|+2s^2\eer which implies $|B|\leq 9$. However, it can be
individually verified that (\ref{er}) cannot hold for $|B|=9$,
completing the claim.
\end{proof}

\begin{claim}\label{useful-claim} If $|B|\geq 7$ and $Z$ is any one-dimensional subspace, then
\ber\label{m-is-big} |\phi_Z(A)|&>&\frac{s|B|}{4},\; \mbox{ when }s\geq 4\\
|\phi_Z(A)|&>&\frac{s|B|}{5},\;\label{m-is-big-for-s=3} \mbox{ when
}s=3.\eer\end{claim}

\begin{proof} Suppose to the contrary that
\ber\label{m-is-too-small} m&\leq&\frac{s|B|}{4},\; \mbox{ when }s\geq 4\\
m&\leq&\frac{s|B|}{5},\;\label{m-is-too-small-even-for-s=3} \mbox{
when }s=3.\eer Note (\ref{m-is-too-small}) and
(\ref{m-is-too-small-even-for-s=3}) each implies $m<|A|$. Let
$l:=\sqrt{\frac{m(m-1)|B|}{|A|-m}}$.

If $s\geq 4$, then (\ref{ideal-A-bound}) and (\ref{m-is-too-small})
imply \be\label{l-is-small}l\leq
\sqrt{\frac{m^2|B|}{\frac{1}{2}s(s-1)|B|+s-m}}<
\sqrt{\frac{s^2|B|^3/16}{\frac{1}{2}s(s-1)|B|-\frac{s|B|}{4}}}=
\frac{|B|}{4}\sqrt{\frac{s^2}{\frac{1}{2}s^2-\frac{3}{4}s}}\leq
\frac{\sqrt{5}}{5}|B|.\ee If $s=3$, then (\ref{ideal-A-bound}) and
(\ref{m-is-too-small-even-for-s=3}) imply
\be\label{l-is-small-3}l\leq
\sqrt{\frac{m^2|B|}{\frac{1}{2}s(s-1)|B|+s-m}}<
\sqrt{\frac{\frac{9}{25}|B|^3}{3|B|-\frac{3}{5}|B|}}\leq
\frac{\sqrt{15}}{10}|B|.\ee

From the proof of Theorem \ref{THEbounds}, we know that $l$
minimizes (\ref{la-lb-bound}), and thus that
(\ref{Bound-unmanipulatedform}) holds with $n'=l$.
%whence applying the estimate (\ref{ideal-A-bound}) yields
%$$(\frac{|B|}{l}-1)m^2+(s|B|-|B|+\frac{|B|}{l}+l-s-2)m+(l-1)(\frac{1}{2}s(s-1)|B|+s)\leq
%0.$$  Consequently, the discriminant of the above quadratic in $m$
%must be nonnegative, i.e.,
%\be\label{the}(s|B|-|B|+M-s-2)^2-2(|B|+1-M)(s^2|B|-s|B|+2s)\geq
%0,\ee where $M:=\frac{|B|}{l}+l$.
%Collecting terms, we obtain
%\be\label{duhdumm}M^2+(2s^2|B|+2s-2|B|-4)M+4+4|B|-4s^2|B|+|B|^2-s^2|B|^2-4s|B|+s^s\geq
%0.\ee Noting that $(2s^2|B|+2s-2|B|-4)\geq 0$, we conclude that
%(\ref{duhdumm}) must hold for the maximum allowed value for $M$.
If $l\leq 3$, then (\ref{la-lb-bound}) will be minimized for either
$n'=1$, $n'=2$ or $n'=3$, whence Claims \ref{n-not-2} and
\ref{n'-not-3} imply $|B|\leq 6$. Note that
$\frac{1}{3}<\max\{\frac{\sqrt{5}}{5},\,\frac{\sqrt{15}}{10}\}$.
Hence if $s\geq 4$, then (\ref{l-is-small}) implies that
\be\label{M-upper} M=\frac{|B|}{l}+l\leq
\frac{5}{\sqrt{5}}+\frac{\sqrt{5}}{5}|B|<\frac{9}{20}|B|+\frac{9}{4},\ee
while if $s=3$, then (\ref{l-is-small-3}) implies that
\be\label{M-upper-3} M=\frac{|B|}{l}+l\leq
\frac{10}{\sqrt{15}}+\frac{\sqrt{15}}{10}|B|<\frac{2}{5}|B|+\frac{13}{5}.\ee
Combining (\ref{M-upper}) and (\ref{duhdumm}) and applying the
estimate $s\geq 4$, we obtain \ber\label{extra-umph} 0&\leq&
-\frac{1}{10}s^2|B|^2-\frac{31}{10}s|B|+\frac{1}{2}s^2|B|+\frac{121}{400}|B|^2
-\frac{11}{40}|B|+\frac{1}{16}+\frac{9}{2}s+s^2\\&\leq&\nn
-\frac{1}{10}s^2|B|^2+\frac{1}{2}s^2|B|+\frac{121}{400}|B|^2+s^2\leq
-(\frac{19}{240}+\frac{1}{48})s^2|B|^2+\frac{1}{2}s^2|B|+\frac{1}{3}|B|^2+s^2\\\nn
&\leq& -\frac{19}{240}s^2|B|^2+\frac{1}{2}s^2|B|+s^2,\eer which
implies $|B|\leq 7$. However, individually checking the case $|B|=7$
in (\ref{extra-umph}) shows that in fact $|B|\leq 6$. Combining
(\ref{M-upper-3}) and (\ref{duhdumm}) and assuming $s=3$, we obtain
$$-36|B|^2+12|B|+624\geq 0,$$ which implies $|B|\leq 4$, completing the claim.
\end{proof}

\begin{claim}\label{A-has-s-collinear-points} There are $s$
collinear points in $A$.\end{claim}

\begin{proof}
Suppose instead that $A$ contains no $s$ collinear points. Then it
follows from the pigeonhole principle and (\ref{ideal-A-bound}) that
\be\label{misbigII} |\phi_Z(A)|> \frac{1}{2}s|B|+1,\ee for any
one-dimensional subspace $Z$. Consequently, if $B$ has at least $3$
collinear points contained in a line parallel to (say) $Z$, then
Theorem \ref{CDT-for-Z} implies $$|A+B|\geq |A|+2|\phi_Z(A)|>
|A|+2(\frac{1}{2}s|B|+1)=|A|+s|B|+2,$$ as desired. Therefore we may
assume $B$ contains no $3$ collinear points.

Suppose $h_1(B,B)<|B|-1$. Then, since $B$ contains no $3$ collinear
points, it follows that there exists a pair of parallel lines each
containing $2$ points of $B$. Hence, by an appropriate affine
transformation, we may w.l.o.g assume
$(0,0),\,(1,0),\,(0,1),(x,1)\in B$, for some $x>0$. Let $x_1=(1,0)$
and $x_2=(0,1)$. Let $A_1\subseteq A$ be the subset obtained by
choosing for each element of $\phi_{Z_1}(A)$ the element of $A$ with
largest $x_1$-coordinate. Let $A_2\subseteq A$ be likewise defined
using $Z_2$ instead of $Z_1$. Note $A_1+(1,0)$ contains
$|\phi_{Z_1}(A)|$ points in $A+B$ disjoint from $A$.

Let $z+\R x_1$ be an arbitrary line parallel to $\R x_1$, and let
$a_1,\ldots,a_r$ be the elements of $A_2\cap (z+\R x_1)$. Moreover,
if $A_1\cap (z+(0,1)+\R x_1)$ is nonempty, then there is a unique
element $y\in A_1\cap (z+(0,1)+\R x_1)$, and so let $a_s,\ldots,a_r$
be those elements of $A_2\cap (z+\R x_1)$ with $\phi_{Z_1}(a_i)\geq
\phi_{Z_1}(y)+1$. If $A_1\cap (z+(0,1)+\R x_1)$ is empty, let
$s=r+1$. Note that for each $a_i$, $i<s$, the element $a_i+(0,1)$ is
an element of $A+B$ contained in neither $A$ nor $A_1+(1,0)$, while
for each $a_i$, $i\geq s$, the element $a_i+(x,1)$ is an element of
$A+B$ contained in nether $A$ nor $A_1+(1,0)$ (since $x>0$).
Consequently, since $z$ is arbitrary and since  $A_1+(1,0)$ contains
$|\phi_{Z_1}(A)|$ points from $A+B$ disjoint from $A$, we conclude
that
$$|A+B|\geq |A+\{(0,0),\,(1,0),\,(0,1),\,(x,1)\}|\geq
|A|+|\phi_{Z_1}(A)|+|\phi_{Z_2}(A)|\geq |A|+s|B|+2,$$ where the
latter inequality follows by (\ref{misbigII}) applied both with
$Z=\R x_1$ and $Z=\R x_2$. Thus (\ref{Goool-Aisbig}) holds, as
desired, and so we may assume $h_1(B,B)=|B|-1$.

Choose $x_1$ such that $|\phi_{Z_1}(B)|<|B|$, and let $A'=\C_X(A)$,
$B'=\C_X(B)$, $A_i=A'\cap (\Z x_1+(i-1)x_2)$ and $B_j=B'\cap (\Z x_1
+(j-1)x_2)$, for $i=1,\ldots,m$ and $j=1,\ldots,n$. Note, since
$h_1(B,B)=|B|-1$ and $|\phi_{Z_1}(B)|<|B|$, that $n=|B|-1$,
$|B_1|=2$, and $|B_i|=1$ for $i>1$. Since $A$ contains no $s$
collinear points, we have $|A_i|\leq s-1$ for all $i$. Observe, for
$j=1,\ldots,m$, that we have the following estimate:
\begin{multline*}|A+B|\geq
\Sum{i=1}{j-1}|A_i+B_1|+\Sum{i=1}{|B|-1}|A_j+B_i|+\Sum{i=j+1}{m}|A_i+B_n|=
\\ |A|+(|B|-2)|A_j|+|B|+(j-1)|B_1|+(m-j)|B_n|-(m+|B|-2)=|A|+(|B|-2)|A_j|+j.
\end{multline*} Thus, assuming (\ref{Goool-Aisbig}) is false, we conclude
that \be\label{telescoping-bounds}  |A_j|\leq
\frac{s(|B|-1)-j-1}{|B|-2}=s+\frac{s-j-1}{|B|-2},\ee for
$j=1,\ldots,m$. Consequently, for $j$ such that $s+(k-1)(|B|-2)\leq
j\leq s+k(|B|-2)-1,$ where $k=1,2,\ldots$, we infer that
\be\label{wee1}|A_j|\leq s-k.\ee Note that \be\label{wee2}|A_j|\leq
s-1\ee for $j=1,\ldots, s-1$, as remarked earlier. Summing
(\ref{wee1}) and (\ref{wee2}) over all possible $j$, we conclude
that $$|A|\leq
(s-1)^2+(|B|-2)\Sum{k=1}{s-1}(s-k)=(s-1)^2+(|B|-2)\frac{s(s-1)}{2}=
\frac{1}{2}s(s-1)|B|-s+1,$$ contradicting (\ref{ideal-A-bound}), and
completing the claim.
\end{proof}

In view of Claim \ref{A-has-s-collinear-points}, choose $x_1$ so
that there are $s$ points on some line parallel to $\Z x_1$. Let
$A'=\C_X(A)$ and $B'=\C_X(B)$. Since $|\phi_{Z_1}(A)|\geq s$ and
since $A$ contains $s$ collinear points on a line parallel to $\Z
x_1$, it follows that $h_1(A',B')\geq h_1(A',A')\geq s$, whence $A'$
and $B'$ also satisfy the hypotheses of the theorem. Furthermore, if
$|A'+B'|\geq |A'|+s(|B'|-1)=|A|+s(|B|-1)$, then the proof is
complete in view of (\ref{well-align-lowerbound-fullycompressed}).
Thus we can w.l.o.g. assume $A=A'$ and $B=B'$ are compressed
subsets.

Let $A_i=A\cap(\Z x_1+(i-1)x_2)$ and $B_j=B\cap (\Z x_1+(j-1)x_2)$
for $i=1,\ldots,m$ and $j=1,\ldots,n$. By the same estimate used for
(\ref{once2}), we have \ber\nn
    |A+B|&\geq&
    |A|+|B|+(n-1)(|A_1|-1)+m(|B_{n}|-1)-|B_{n}|\\ &\geq&
    |A|+|B|-1+(n-1)(|A_1|-1)\label{wethree}.\eer

If $|B|\geq 9$, then Claims \ref{n-not-2}, \ref{n'-not-3} and
\ref{n'-not-4} imply $n\geq 5$, whence Claim \ref{useful-claim} and
(\ref{wethree}) imply that
$$|A+B|\geq |A|+|B|-1+4(\frac{s|B|+1}{4}-1)=|A|+(s+1)|B|-4,$$ if
$s\geq 4$, and that $$|A+B|\geq
|A|+|B|-1+4(\frac{3|B|+1}{5}-1)=|A|+\frac{17}{5}|B|-\frac{21}{5}>|A|+3|B|-1,$$
if $s=3$. In both cases (\ref{Goool-Aisbig}) follows, as desired. So
we may assume $|B|\leq 8$. In view of Claim \ref{n-not-2} applied
with $Z=\Z x_1$ and $Z=\Z x_2$, we infer that $|B|\geq 5$.

Using the estimate from (\ref{(l+1)/2-bound}) (with the roles of $A$
and $B$ reversed), we obtain
$$|A|+|B|-1+(|B_1|-1)\frac{m-1}{2}+(n-1)(\frac{|A|}{m}-1)\leq
|A+B|\leq |A|+s|B|-s-1.$$ Multiplying by $m$, applying
(\ref{ideal-A-bound}), and rearranging terms yields
$$\frac{|B_1|-1}{2}\cdot m^2-
(s|B|-|B|+\frac{|B_1|-3}{2}+n-s)m+(n-1)(\frac{1}{2}s(s-1)|B|+s)\leq
0.$$ Consequently, the discriminant of the above quadratic in $m$
must be nonnegative, implying \be\label{small-cases}
(s|B|-|B|+\frac{|B_1|-3}{2}+n-s)^2-(|B_1|-1)(n-1)(s(s-1)|B|+2s)\geq
0\ee

If $|B|=5$, then from Claim \ref{n-not-2}, applied with $Z=\Z x_1$
and $Z=\Z x_2$, we conclude $n=|B_1|=3$. Thus (\ref{small-cases})
implies $4s^2+4s-4\leq 0$, contradicting $s\geq 3$. If $|B|=7$, then
from Claims \ref{n-not-2} and \ref{n'-not-3}, applied with $Z=\Z
x_1$ and $Z=\Z x_2$, we conclude $n=|B_1|=4$. Thus
(\ref{small-cases}) implies $27s^2-15s-\frac{25}{4}\leq 0$,
contradicting $s\geq 3$. If $|B|=8$, then from Claims \ref{n-not-2}
and \ref{n'-not-3}, applied with $Z=\Z x_1$ and $Z=\Z x_2$, we
conclude $n\geq 4$ and $|B_1|\geq 4$. Thus (\ref{small-cases})
implies $23s^2-5s-\frac{49}{4}\leq 0$, contradicting $s\geq 3$.
Consequently, it remains only to handle the case $|B|=6$.

In view of Claim \ref{n-not-2} and by swapping the roles of $x_1$
and $x_2$ if necessary, we may assume $n=3$. Hence
(\ref{la-lb-bound}) implies that (\ref{Bound-unmanipulatedform})
holds with $n'=3$. Thus considering
(\ref{Bound-m-quadform-original}) as a quadratic in $m$, we conclude
that the discriminant is nonnegative, i.e., that
\ber|A|&\leq&\label{case6-partduo}\frac{(5s-3)^2}{8}=
\frac{1}{8}(2s-1)^2|B|-\frac{1}{4}(2s-1)+\frac{(s-1)^2}{2(|B|-2)}.\eer
This completes the proof in case (a) holds. From (\ref{the}), we
have \be\label{case6-parttree}0\leq (5s-3)^2-24s^2+16s=s^2-14s+9,\ee
which implies $s\geq 14$. Thus $|B|\geq \frac{2s+4}{3}\geq
\frac{32}{3}>6$, contradicting the hypothesis of (b), and completing
the proof.
\end{proof}


\begin{thebibliography}{9}

\bibitem{bilu} Y.~Bilu, Structure of sets with small sumset,
Structure theory of set addition, \emph{Ast\'{e}risque} (1999)  no.
{\bf 258}, xi, 77--108.

\bibitem{BL} B.~Bollob\'as and I.~Leader, Sums in the grid,
{\it Discrete Math.} {\bf 162} (1996), no. 1-3, 31–-48.

\bibitem{Ch}
 G.~Chang, A polynomial bound in Freiman's Theorem,
  {\it Duke Math J.} {\bf 113} (1994), no. 3, 399–-419.

\bibitem{Fi}
  P.~Fishburn, On a contribution of Freiman to additive
  number  theory, {\it J. Number Teory} {\bf 35} (1990) 325--334.


\bibitem{F1}
  G.~Freiman, \emph{Foundations of Structural Theory of Set Addition},
  Transl. Math. Monographs {\bf 37}, Amer. Math. Soc. Providence, R.I.
  1973.

\bibitem{F2}
  G.~Freiman, What is the structure of $K$ if $K+K$ is small?,
  \emph{Lecture Notes in Mathematics} {\bf 1240}, Springer, New York, 1987,
  189--134.

\bibitem{Ga} R.J.~Gardner and P.~Gronchi, A Brunn--Minkowski
inequality for the integer lattice, {\it Trans. Amer. Math. Soc.}
{\bf 353} (2001), no. 10, 3995--4024.

\bibitem{GT} B.~Green and T.~Tao,  Compressions, Convex Geometry
and the Freiman--Bilu Theorem, {\it arXiv:math:NT/0511069 v2}, 3
March 2006.

\bibitem{K} D.~Kleitman, Extremal hypergraph problems, \emph{Surveys
in Combinatorics}, ed. by B.~Bollob\'{a}s, Cambridge University
Press, Cambridge, 1979, 44--65.

\bibitem{LS} S.~Lev and P.~Y.~Smeliansky, On addition of two
sets of integers, {\it Acta Arith.} {\bf 70} (1995), 85--91.

\bibitem{Natbook} M.~Nathanson, \emph{Additive number theory:
Inverse problems and the geometry of sumsets}, Graduate Texts in
Mathematics {\bf 165}, Springer-Verlag, New York, 1996.

\bibitem{rudin} W. Rudin, \emph{Real and Complex Analysis}, WCB
McGraw Hill, San Francisco, 1987.

\bibitem{R} I.~Ruzsa, Sum of sets in several dimensions,
\emph{Combinatorica} {\bf 14} (1994), 485--490.

\bibitem{S2}
  Y.~Stanchescu,
  On the structure of sets with small doubling property on the plane
  (I), {\it Acta Arith.} {\bf 83} (2) (1998) 127--141.

\bibitem{taobook} T.~Tao and V.~Vu, \emph{Additive combinatorics},
Cambridge Studies in Advanced Mathematics {\bf 105}, Cambridge
University Press, Cambridge, 2006.

\end{thebibliography}
\end{document}